\documentclass[11pt]{amsart}

\usepackage[dvips]{color}

\usepackage{amsmath}

\usepackage{amsxtra}

\usepackage{amscd}

\usepackage{amsthm}

\usepackage{amsfonts}
\usepackage{amssymb}
\usepackage{eucal}
\usepackage{epsfig}
\usepackage{graphics}
\theoremstyle{plain}
\newtheorem{theorem}{Theorem}[section]

\newtheorem{thm}[theorem]{Theorem}
\newtheorem{cor}[theorem]{Corollary}
\newtheorem{prop}[theorem]{Proposition}
\newtheorem{lem}[theorem]{Lemma}
\newtheorem{conj}[theorem]{Conjecture}
\newtheorem{defi}[theorem]{Definition}
\newtheorem{example}[theorem]{Example}

\newcommand{\Glie}{\mathfrak{g}}
\newcommand{\Yim}{\mathcal{Y}}

\newcommand{\ZZ}{\mathbb{Z}}
\newcommand{\CC}{\mathbb{C}}

\newcommand{\QQ}{\mathbb{Q}}

\newcommand{\C}{\mathbb{C}}

\newcommand{\Z}{\mathbb{Z}}

\newcommand{\g}{\mathfrak{g}}
\newcommand{\bo}{\mathfrak{b}}
\newcommand{\tb}{\mathbf{\mathfrak{t}}}
\newcommand{\ga}{\overline{\alpha}}
\newcommand{\gb}{\dot{\mathfrak{g}}}

\newcommand{\Psib}{\mbox{\boldmath$\Psi$}}
\newcommand{\Psibs}{\scalebox{.7}{\boldmath$\Psi$}}
\newcommand{\qbin}[2]{{\left[
\begin{matrix}{\,\displaystyle #1\,}\\
{\,\displaystyle #2\,}\end{matrix}
\right]
}}

\newcommand{\nc}{\newcommand}
\nc{\on}{\operatorname}
\nc{\la}{\lambda}
\nc{\wh}{\widehat}

\newtheorem{rem}[theorem]{Remark}

\begin{document}

\begin{title}
{Baxter's relations and spectra of quantum integrable models}
\end{title}

\author[Edward Frenkel]{Edward Frenkel}

\address{Department of Mathematics, University of California,
  Berkeley, CA 94720, USA}

\author[David Hernandez]{David Hernandez}

\address{Sorbonne Paris Cit\'e, Univ Paris Diderot, CNRS Institut de
  Math\'ematiques de Jussieu-Paris Rive Gauche UMR 7586,
B\^atiment Sophie Germain, Case 7012,
75205 Paris Cedex 13, France.}

\begin{abstract} Generalized Baxter's relations on the
  transfer-matrices (also known as Baxter's $TQ$ relations) are
  constructed and proved for an arbitrary untwisted quantum affine
  algebra. Moreover, we interpret them as relations in the
  Grothendieck ring of the category $\mathcal{O}$ introduced by Jimbo
  and the second author in \cite{HJ} involving infinite-dimensional
  representations constructed in \cite{HJ}, which we call here
  ``prefundamental''. We define the transfer-matrices associated to
  the prefundamental representations and prove that their eigenvalues
  on any finite-dimensional representation are polynomials up to a
  universal factor. These polynomials are the analogues of the
  celebrated Baxter polynomials. Combining these two results, we
  express the spectra of the transfer-matrices in the general quantum
  integrable systems associated to an arbitrary untwisted quantum
  affine algebra in terms of our generalized Baxter polynomials. This
  proves a conjecture of Reshetikhin and the first author formulated
  in 1998 \cite{Fre}. We also obtain generalized Bethe Ansatz
  equations for all untwisted quantum affine algebras.
\end{abstract}

\dedicatory{To Victor Kac on his birthday}

\maketitle

\setcounter{tocdepth}{1}
\tableofcontents

\section{Introduction}

The partition function $Z$ of a quantum model on an $M \times N$ lattice
may be written in terms of the eigenvalues of the row-to-row transfer
matrix $T$:
$$
Z = \on{Tr} T^M = \sum_i \la_i^M.
$$
Therefore, to find $Z$, one needs to find the spectrum of $T$.

In his seminal 1971 paper \cite{Baxter}, R. Baxter tackled this
question for the so-called eight-vertex model, in which $T$ acts on
the vector space $(\C^2)^{\otimes N}$. In the special case that the
parameters satisfy the ``ice condition'' (then it is called the
six-vertex model) the spectrum of the model was previously found by
E. Lieb \cite{Lieb1,Lieb2,Lieb3} (see also \cite{Sut}) using
%there is one obvious eigenvector: $(v_+)^{\otimes
%  N}$, where $\{ v_+,v_- \}$ is a basis in $\C^2$, and the eigenvalue
%may be written in the form $A(z) + D(z)$
an explicit construction of eigenvectors now referred to as Bethe
Ansatz. Analyzing this result, Baxter observed that the eigenvalues
of $T$ on these eigenvectors always have the form
\begin{equation}    \label{relB}
A(z) \frac{Q(zq^2)}{Q(z)} + D(z) \frac{Q(zq^{-2})}{Q(z)},
\end{equation}
where $Q(z)$ is a polynomial, $z,q$ are parameters of the model,
and the functions $A(z), D(z)$ are the same for all eigenvalues.
%(for example, to the eigenvector $(v_+)^{\otimes N}$ corresponds
%$Q(z)=1$).
Furthermore, Baxter realized that the condition that the seeming poles
of the above expression, occurring at the roots of $Q(z)$, cancel each
other is equivalent to the Bethe Ansatz equations guaranteeing that
the vectors constructed by the Bethe Ansatz are indeed
eigenvectors. Thus, apart from the factors $A(z)$ and $D(z)$ which are
universal, the spectrum of $T$ is essentially determined by the
polynomials $Q(z)$ satisfying this condition (provided that the Bethe
Ansatz gives us all eigenvectors).\footnote{In his paper, Baxter went
  on to generalize this relation to the general eight-vertex model,
  for which Bethe Ansatz was not available, but this is beyond the
  scope of the present paper.}

The polynomial $Q(z)$ is now called {\em Baxter's polynomial}, and
relation (\ref{relB}) is called {\em Baxter's relation} (or Baxter's
$TQ$ relation). It looks rather mysterious. Why should such a relation
hold?

To gain insights into this question, we present a modern
interpretation of Baxter's result in a broader context of quantum
groups. Consider the quantum affine algebra $U_q ({\mathfrak g})$
associated to an untwisted affine Kac--Moody algebra. The completed
tensor square of this algebra contains the universal $R$-matrix
${\mathcal R}$ satisfying the Yang--Baxter relation and other
properties. Given a finite-dimensional representation $V$ of $U_q
({\mathfrak g})$, we construct the transfer-matrix
$$
t_V(z) = \on{Tr}_V (\pi_{V(z)} \otimes \on{id})({\mathcal R}),
$$
where $V(z)$ is a twist of $V$ by a ``spectral parameter'' $z$. It
turns out that
$$
[t_V(z),t_{V'}(z')] = 0
$$
for all $V, V'$ and $z,z'$. Therefore these transfer-matrices give
rise to a family of commuting operators on any finite-dimensional
representation $W$ of $U_q ({\mathfrak g})$.

In the special case that ${\mathfrak g} = \wh{sl}_2$, $V$ a simple
two-dimensional representation of $U_q (\widehat{sl}_2)$, and $W$
the tensor product of $N$ two-dimensional representations, the
operator $t_V(z)$ acting on $W$ becomes Baxter's transfer-matrix. This
makes it clear that an analogue of Baxter's problem may be formulated
for an arbitrary quantum affine algebra $U_q({\mathfrak g})$. Namely,
it is the problem of describing the eigenvalues of the
transfer-matrices $t_V(z)$ on finite-dimensional representations $W$
of $U_q({\mathfrak g})$. It is known that these eigenvalues appear as
the spectra of quantum integrable systems generalizing the six-vertex
model (more precisely, generalizing the $XXZ$ model, whose spectrum is
the same as that of the six-vertex model). Hence a solution of this
problem has immediate applications in statistical mechanics.

In \cite{Fre}, N. Reshetikhin and the first author found a novel and
general way to describe the eigenvalues of the transfer-matrices for
an arbitrary (untwisted) quantum affine algebra, generalizing Baxter's
formula. The idea was to use the $q$-characters of finite-dimensional
representations of quantum affine algebras introduced in \cite{Fre}
(note that a similar notion for representations of the Yangians was
introduced earlier by H. Knight). The $q$-character is a homomorphism
of rings
$$
\chi_q: \on{Rep} \; U_q ({\mathfrak g}) \to \Z[Y_{i,a}^{\pm
  1}]_{i \in I, a \in \C^\times},
$$
where $I$ is the set of vertices of the Dynkin diagram of the
finite-dimensional simple Lie algebra underlying ${\mathfrak
  g}$. For example, if ${\mathfrak g} = \wh{sl}_2$ and $V$ is a
two-dimensional representation, then there is $a\in\CC^*$ such that
$$
\chi_q(V) = Y_{1,a} + Y_{1,aq^2}^{-1}.
$$
Roughly speaking, the point is that the above Baxter relation (after
renaming the variables $z \mapsto aq$) may be obtained from this
formula if we substitute
$$
Y_{1,a} \mapsto \frac{Q(aq^{-1})}{Q(aq)}
$$
(to simplify matters, we are dropping the factors $A(z)$ and $D(z)$
for now; but they can be easily taken into account). This gives us a
way to generalize Baxter's formula.

Namely, the following conjecture was proposed in \cite{Fre}: Given a
finite-dimensional representation $V$ of $U_q ({\mathfrak g})$, all of
the eigenvalues of $t_V(a)$ on any irreducible finite-dimensional
representation $W$ may always be written in the following form: we
take the $q$-character $\chi_q(V)$ and substitute in it
$$
Y_{i,a} \mapsto \frac{Q_{i,aq_i^{-1}}}{Q_{i,aq_i}}, \qquad i \in I,
$$
where $Q_{i,a}$ is the product of two factors: one of them is the same
for all eigenvalues (it depends only on $W$) and the other is a
polynomial -- these are the analogues of Baxter's polynomial. (Here
$q_i = q^{d_i}$, see Section \ref{debut}; the precise statement is
in Theorem \ref{fdcase}.)

We remark that in various special cases, a similar conjectural
description of the eigenvalues of the transfer-matrices was proposed
by N. Reshetikhin \cite{R1,R2,R3}; V. Bazhanov and N. Reshetikhin
\cite{BR}; and A. Kuniba and J. Suzuki \cite{KS}.

\smallskip

In this paper we prove the general conjecture of \cite{Fre} about the
eigenvalues of the transfer-matrices in a deformed setting (this means
that the trace used in the above formula for the transfer-matrix is
replaced by the twisted trace that depends on additional parameters
$u_i, i \in I$; see Definition \ref{gr tr}). Among other things, our
proof gives a conceptual explanation of Baxter's relation, and its
generalizations, in terms of representation theory of quantum affine
algebras.

\smallskip

Our proof is based on two results which are of independent interest.

\smallskip

First, we show that the above $Q_{i,a}$ is itself an eigenvalue of a
transfer-matrix (a generalization of Baxter's $Q$-operator) -- the one
associated to what we call here the $i$th {\em prefundamental}
representation $L^+_{i,a}$. This is an infinite-dimensional
representation of the Borel subalgebra (in the Kac--Moody sense) $U_q
({\mathfrak b})$ of $U_q({\mathfrak g})$ that was introduced by
M. Jimbo and the second author in \cite{HJ}.

In order to explain what it is, recall the classification of
irreducible finite-dimensional representations of $U_q({\mathfrak g})$
due to Drinfeld \cite{Dri2} and Chari--Pressley \cite{CP}. The algebra
$U_q ({\mathfrak g})$ has loop generators $x^\pm_{i,n}, i \in I, n \in
\Z$; $h_{i,n}, i \in I, n \neq 0$; and $k_i^{\pm 1}, i \in I$. Each
irreducible finite-dimensional representation is generated by a
``highest weight vector'', that is, a vector annihilated by
$x^+_{i,n}, i \in I, n \in \Z$, which is an eigenvector of the loops
to Cartan generators $h_{i,n}, i \in I, n \neq 0$ and $k_i^{\pm
  1}$. Furthermore, the eigenvalue of their generating function
$$
\phi^\pm_i(z) = k^{\pm 1}_i \exp \left( \pm (q_i-q_i^{-1}) \sum_{n>0}
h_{i,\pm n} z^{\pm n} \right)
$$
is the expansion in $z^{\pm 1}$ of the rational function $q_i^{\on{deg}
  P_i} P_i(zq_i^{-1})/P_i(zq_i)$, where $P_i(z)$ is a polynomial with
constant term $1$. These polynomials, called Drinfeld polynomials,
record the ``highest $\ell$-weight'' of the representation.

In \cite{HJ}, M. Jimbo and the second author extended the category of
finite-dimensional representations of $U_q ({\mathfrak g})$ to a
category denoted by ${\mathcal O}$. This is a category of (possibly
infinite-dimensional) representations of $U_q ({\mathfrak b})$ which
have weight decomposition with respect to the finite-dimensional
Cartan subalgebra generated by $k_i^{\pm 1}$ and such that all weight
components are finite-dimensional. It was shown in \cite{HJ} that
irreducible representations of this category are also generated by
highest weight vectors (in the above sense), but the corresponding
highest $\ell$-weights (the eigenvalues of $\phi^\pm_j(z)$) are given
by arbitrary rational functions which are regular and non-zero at the
origin (but may have a zero or a pole at infinity).

The $i$th prefundamental representation $L^+_{i,a}$ is then by
definition the representation for which the eigenvalue of
$\phi^\pm_j(z)$ on the highest weight vector is equal to $1$ if $j
\neq i$ and to $1-za$ if $j=i$.

As far as we know, such representations were first constructed in the
case of ${\mathfrak g} = \wh{sl}_2$ by V. Bazhanov, S. Lukyanov, and
A. Zamolodchikov \cite{BLZ,BLZ1}. Their construction was subsequently
generalized to ${\mathfrak g} = \wh{sl}_3$ in \cite{BHK}, and to
${\mathfrak g} = \wh{sl}_{n+1}$ with $i = 1$ in \cite{Ko}. For general
${\mathfrak g}$, the prefundamental representations were constructed
in \cite{HJ}.

A marvelous insight of \cite{BLZ,BLZ1} was the identification of the
transfer-matrix of this representation in the case of ${\mathfrak g} =
\wh{sl}_2$ with the Baxter operator. From the point of view discussed
above, this enables one to interpret Baxter's $TQ$ relation as a
relation in the Grothendieck ring of the category $\mathcal{O}$. Here
we generalize this result to all untwisted quantum affine algebras.

\medskip

Namely, we establish the following relation in the Grothendieck ring
of ${\mathcal O}$, generalizing the Baxter relation: for any
finite-dimensional representation $V$ of $U_q ({\mathfrak g})$, take
its $q$-character and replace each $Y_{i,a}$ by the ratio of the
classes of prefundamental representations
$[L^+_{i,aq_i^{-1}}]/[L^+_{i,aq_i}]$ times the class of the
one-dimensional representation $[\omega_i]$ of $U_q ({\mathfrak b})$
on which the finite-dimensional Cartan subalgebra acts according to
the $i$th fundamental weight. Then this expression is equal to the
class of $V$ in the Grothendieck ring of ${\mathcal O}$ (viewed as a
representation of $U_q ({\mathfrak b})$ obtained by restriction from
$U_q ({\mathfrak g})$). This is our first main result.

For example, if ${\mathfrak g} = \wh{sl}_2$ and $V$ is the
two-dimensional representation, then we have
\begin{equation}    \label{rel1}
[V] = [\omega_1] \frac{[L^+_{1,aq^{-1}}]}{[L^+_{1,aq}]} +
[-\omega_1] \frac{[L^+_{1,aq^{3}}]}{[L^+_{1,aq}]},
\end{equation}
or equivalently,
$$
[V] [L^+_{1,aq}] = [\omega_1] [L^+_{1,aq^{-1}}] +
[{-\omega}_1] [L^+_{1,aq^{3}}],
$$
which follows from the fact that $V \otimes L^+_{1,aq}$ is an
extension of two representations: $[\omega_1] \otimes L^+_{1,aq^{-1}}$
and $[{-\omega}_1] \otimes L^+_{1,aq^{3}}$ (see \cite[Section
2]{jms}).

\medskip

Our second main result is that the (twisted) transfer-matrix
associated to $L^+_{i,a}$ is well-defined (despite the fact that
$L^+_{i,a}$ is infinite-dimensional), and further, all of its
eigenvalues on any irreducible finite-dimensional representation $W$
of $U_q ({\mathfrak g})$ are {\em polynomials} up to one and the same
factor that depends only on $W$ (more precisely, we prove this for the
prefundamental representations in the category dual to
$\mathcal{O}$). Denoting these eigenvalues by $Q_{i,a}$, and combining
our two results, we obtain the proof of the conjecture of Reshetikhin
and the first author.

For example, if ${\mathfrak g} = \wh{sl}_2$ and $V$ is the
two-dimensional representation, then \eqref{rel1} implies the Baxter
equation \eqref{relB} for the eigenvalues of the transfer-matrix of
$V$ (after renaming the variables $a \mapsto zq^{-1}$).

As explained in \cite[Section 6.3]{Fre} and in Section \ref{bae}
below, the formula for the eigenvalues of the transfer-matrices in
terms of the polynomials $Q_{i,a}$ leads to a natural system of
equations on the roots of these polynomials. These equations ensure
the cancellation of the apparent poles in the eigenvalues of the
transfer-matrices due to the appearance of the polynomials $Q_{i,a}$
in the denominator. These are the generalized Bethe Ansatz equations,
which are suitably modified equations (6.6) of \cite{Fre} (they are
modified because we use the twisted trace in the definition of the
transfer-matrices, which depends on additional parameters). We
conjecture that the solutions of these generalized Bethe Ansatz
equations are in one-to-one correspondence with the eigenvalues of the
(twisted) transfer-matrices.

If ${\mathfrak g} = \wh{sl}_2$ and $V$ is the two-dimensional
representation, these equations are precisely the Bethe Ansatz
equations of the six-vertex model. This observation was used by
N. Reshetikhin as a guiding principle for writing conjectural Bethe
Ansatz equations and the eigenvalues of the transfer-matrices for some
${\mathfrak g}$ \cite{R1,R2,R3} -- a procedure he dubbed ``analytic
Bethe Ansatz'' (see also \cite{BR,KS}). The results of \cite{Fre} and
the present paper give us a conceptual explanation of this procedure.

\medskip

We close this Introduction with the following three remarks.

\smallskip

(1) Our results show that the prefundamental representations have an
important role to play in representation theory of quantum affine
algebra. They are infinite-dimensional, but in many ways they have a
simpler structure and behavior than finite-dimensional
representations. And they can be used effectively to prove results
about finite-dimensional representations that were previously out of
reach, such as the conjecture on the spectra of transfer-matrices
that we prove in this paper. As another application, we use our
results on the polynomiality of the transfer matrix of $L^+_{i,a}$ to
show that a certain generating function of the Drinfeld's Cartan
elements $h_{i,n}$ is, up to a universal factor, a polynomial on any
finite-dimensional irreducible representation of $U_q ({\mathfrak
  g})$.

% As a corollary, we also prove the polynomiality of the Drinfeld's
% loop to Cartan generators.

%\smallskip

(2) As we learned from Nikita Nekrasov and Andrei Okounkov, Baxter's
polynomials $Q_{i,a}$ should have a geometric
interpretation. Finite-dimensional representations of $U_q ({\mathfrak
  g})$ may be realized in equivariant $K$-theory of quiver varieties
as shown by H. Nakajima \cite{Nak}, and similarly, finite-dimensional
representations of the Yangian are realized in equivariant cohomology
of these varieties \cite{Var} (see also \cite{mo}). In the Yangian
case, the Baxter polynomial is expected to be equal to the operator of
(quantum) multiplication by the Chern polynomial of a certain
tautological vector bundle on the quiver variety (see \cite{NS},
p.~15), and there is a similar conjecture in the case of $U_q
({\mathfrak g})$. It would be interesting to connect our results with
the geometry of quiver varieties. In particular, it is an interesting
question to find a geometric realization of the prefundamental
representations analogous to that of finite-dimensional
representations of $U_q ({\mathfrak g})$.

We remark that the category of finite-dimensional representations of $U_q
({\mathfrak g})$ is equivalent to that of the Yangian $Y_h(({\mathfrak
  g})$ (where $q=e^{\pi i h}$ and $h$ is not a rational number), as
shown by S. Gautam and V. Toledano Laredo \cite{GTL}.

%\smallskip

(3) After this paper was finished, we learned from N. Nekrasov about
his joint work with V. Pestun and S. Shatashvili (subsequently
published as \cite{NPS}), in which the $q$-characters are used to
describe quantum geometry of the $\Omega$-deformations of 5D
supersymmetric quiver gauge theories (and similarly for 4D theories,
with quantum affine algebras replaced by Yangians). Note however that
in \cite{NPS} the analytic properties of the functions $Q_{i,a}$ and
$t_V(a)$ are quite different from ours. Thus, it seems that the
$q$-characters represent a rather general algebraic structure that can
be used (by imposing various analytic conditions on $Q_{i,a}$ and
$t_V(a)$) to describe not only the models of statistical mechanics
discussed in the present paper, but also other models of quantum
physics.

As explained in \cite{Fre}, the $q$-character map is a $q$-deformation
of the Miura transformation. And so the $q$-characters play a role
similar to that of the Miura transformation in the Gaudin model and
its generalizations: describing the spectra of the Hamiltonians of the
model \cite{FFR,F:icmp}. Moreover, the Miura transformation and the
$q$-character map arise naturally from the center of a completed
enveloping algebra of an affine Kac--Moody algebra, as explained in
\cite{FFR,F:icmp} (resp., quantum affine algebra, as explained in
\cite{Fre}) at the critical level. Thus, it is this center that
is the fundamental algebraic object governing a large class of
quantum integrable models.

\medskip

The paper is organized as follows. In Section \ref{bck} we recall the
definitions and the main properties of quantum affine (or loop)
algebras and the corresponding Borel subalgebras. In Section
\ref{catO} we recall important results about their representations; in
particular, finite-dimensional representation as well as those from
the category $\mathcal{O}$ (such as the prefundamental
representations), introduced in \cite{HJ}. In Section \ref{bacato} we
prove a uniform explicit $q$-character formula for positive
prefundamental representations (Theorem \ref{formuachar}).  We prove
that it implies our first main result: the realization of generalized
Baxter's relations in the Grothendieck ring of $\mathcal{O}$ (Theorem
\ref{relations}). We also prove that an arbitrary tensor product of
positive prefundamental representations is simple (Theorem
\ref{stensor}).  In Section \ref{trpol} we state our second main
result: polynomiality of the twisted transfer-matrices associated to
prefundamental representations (Theorem \ref{catOpol}). Our main
application is the proof of a deformed version of the conjecture of
Reshetikhin and the first author (Theorem \ref{fdcase}). We use this
result to write down the system of Bethe Ansatz equations explicitly
in Section \ref{bae}. We also derive the polynomiality of Drinfeld's
Cartan elements on finite-dimensional representations (Theorem
\ref{jordan}) and prove commutativity of the twisted transfer-matrices
associated to representations in the category $\mathcal{O}$ (Theorem
\ref{com}).  In Section \ref{gradation}, we establish the existence of
a certain grading on positive prefundamental representations (Theorem
\ref{defigrad}). This result is used in Section \ref{endproof} to
conclude the proof of Theorem \ref{catOpol}.

\medskip

{\bf Acknowledgments.} We are grateful to Nikita Nekrasov and Andrei
Okounkov for valuable discussions. We also thank Ilaria Damiani for
communications on root vectors.

This paper was completed while the first author was visiting
Universit\'e Paris-Diderot Paris 7, which he thanks for
hospitality. He also acknowledges support of the NSF.

The results of this paper were presented at the conference
``Symmetries in Mathematics and Physics II'' in honor of Victor Kac's
70th birthday at IMPA (Rio de Janeiro) in June 2013. It is a pleasure
to dedicate this paper to him.

\section{Quantum loop algebra and Borel algebras}\label{bck}

\subsection{Quantum loop algebra}\label{debut}

Let $C=(C_{i,j})_{0\leq i,j\leq n}$ be an indecomposable Cartan matrix
of untwisted affine type.  We denote by $\g$ the Kac--Moody Lie
algebra associated with $C$.  Set $I=\{1,\ldots, n\}$, and denote by
$\gb$ the finite-dimensional simple Lie algebra associated with the
Cartan matrix $(C_{i,j})_{i,j\in I}$.  Let $\{\alpha_i\}_{i\in I}$,
$\{\alpha_i^\vee\}_{i\in I}$, $\{\omega_i\}_{i\in I}$,
$\{\omega_i^\vee\}_{i\in I}$, $\dot{\mathfrak{h}}$ be the simple
roots, the simple coroots, the fundamental weights, the fundamental
coweights and the Cartan subalgebra of $\gb$, respectively.  We set
$Q=\oplus_{i\in I}\Z\alpha_i$, $Q^+=\oplus_{i\in I}\Z_{\ge0}\alpha_i$,
$P=\oplus_{i\in I}\Z\omega_i$.  Let $D=\mathrm{diag}(d_0\ldots, d_n)$
be the unique diagonal matrix such that $B=DC$ is symmetric and
$d_i$'s are relatively prime positive integers.  We denote by
$(~,~):Q\times Q\to\Z$ the invariant symmetric bilinear form such that
$(\alpha_i,\alpha_i)=2d_i$.  We use the numbering of the Dynkin
diagram as in \cite{ka}.  Let $a_0,\cdots,a_n$ stand for the Kac label
(\cite{ka}, pp.55-56). We have $a_0 = 1$ and we set $\alpha_0 =
-(a_1\alpha_1 + a_2\alpha_2 + \cdots + a_n\alpha_n)$.

Throughout this paper, we fix a non-zero complex number $q$ which is
not a root of unity.  We set $q_i=q^{d_i}$. We also set $h\in\CC$ such
that $q = e^h$, so that $q^r$ is well-defined for any $r\in\QQ$. We
will use the standard symbols for $q$-integers
\begin{align*}
[m]_z=\frac{z^m-z^{-m}}{z-z^{-1}}, \quad
[m]_z!=\prod_{j=1}^m[j]_z,
 \quad 
\qbin{s}{r}_z
=\frac{[s]_z!}{[r]_z![s-r]_z!}. 
\end{align*}
We will use the quantum Cartan matrix $C(q)= (C_{i,j}(q))_{i,j\in I}$
defined by $C_{i,j}(q) = [C_{i,j}]_q$ if $i\neq j$ in $I$ and
$C_{i,i}(q) = [2]_{q_i}$ for $i\in I$. The symmetrized quantum Cartan
matrix $B(q) = (B_{i,j}(q))_{i,j\in I}$ is defined by $B_{i,j}(q) =
[d_i]_qC_{i,j}(q)$ for $i,j\in I$. We denote by $\tilde{B}(q)$
(resp. $\tilde{C}(q)$) the inverse of $B(q)$ (resp. $C(q)$).

The quantum loop algebra $U_q(\g)$ is the $\C$-algebra defined by generators 
$e_i,\ f_i,\ k_i^{\pm1}$ ($0\le i\le n$) 
and the following relations for $0\le i,j\le n$.
\begin{align*}
&k_ik_j=k_jk_i,\quad k_0^{a_0}k_1^{a_1}\cdots k_n^{a_n}=1,\quad
&k_ie_jk_i^{-1}=q_i^{C_{i,j}}e_j,\quad k_if_jk_i^{-1}=q_i^{-C_{i,j}}f_j,
\\
&[e_i,f_j]
=\delta_{i,j}\frac{k_i-k_i^{-1}}{q_i-q_i^{-1}},
\\
&\sum_{r=0}^{1-C_{i.j}}(-1)^re_i^{(1-C_{i,j}-r)}e_j e_i^{(r)}=0\quad (i\neq j),
&\sum_{r=0}^{1-C_{i.j}}(-1)^rf_i^{(1-C_{i,j}-r)}f_j f_i^{(r)}=0\quad (i\neq j)\,.
\end{align*}
Here we have set $x_i^{(r)}=x_i^r/[r]_{q_i}!$ ($x_i=e_i,f_i$). 
The algebra $U_q(\g)$ has a Hopf algebra structure given by
\begin{align*}
&\Delta(e_i)=e_i\otimes 1+k_i\otimes e_i,\quad
\Delta(f_i)=f_i\otimes k_i^{-1}+1\otimes f_i,
\quad
\Delta(k_i)=k_i\otimes k_i\,,
\\
&
S(e_i) = -k_i^{-1} e_i ,\quad S(f_i) = -f_i k_i,
\quad
S(k_i)=k_i^{-1}\,,
\end{align*}
where $i=0,\cdots,n$. 

The algebra $U_q(\g)$ can also be presented in terms of the Drinfeld
generators \cite{Dri2,bec}
\begin{align*}
  x_{i,r}^{\pm}\ (i\in I, r\in\Z), \quad \phi_{i,\pm m}^\pm\ (i\in I,
  m\geq 0), \quad k_i^{\pm 1}\ (i\in I).
\end{align*}
\begin{example} In the case $\gb = sl_2$, we have $e_1 = x_{1,0}^+$, $e_0 =
  k_1^{-1} x_{1,1}^-$, $f_1 = x_{1,0}^-$ and $f_0 = x_{1,-1}^+ k_1$.
\end{example}
We shall use the generating series $(i\in I)$: 
$$\phi_i^\pm(z) = \sum_{m\geq 0}\phi_{i,\pm m}^\pm z^{\pm m} =
k_i^{\pm 1}\text{exp}\left(\pm (q_i - q_i^{-1})\sum_{m > 0} h_{i,\pm
    m} z^{\pm m} \right).$$
We also set $\phi_{i,\pm m}^\pm = 0$ for $m < 0$, $i\in I$.
\begin{rem} In \cite{Fre}, the notation $h_{i,m}$ is used for
  $[d_i]_qh_{i,m}$. \end{rem}

The algebra $U_q(\g)$ has a $\ZZ$-grading defined by $\on{deg}(e_i)
= \on{deg}(f_i) = \on{deg}(k_i^{\pm 1}) = 0$ for $i\in I$ and
$\on{deg}(e_0) = - \on{deg}(f_0) = 1$.  It satisfies
$\on{deg}(x_{i,m}^\pm) = \on{deg}(\phi_{i,m}^\pm) = m$ for $i\in
I$, $m\in\ZZ$. For $a\in\CC^*$, there is a corresponding automorphism
$\tau_a : U_q(\g)\rightarrow U_q(\g)$ such an element $g$ of degree
$m\in\ZZ$ satisfies $\tau_a(g) = a^m g$.

The algebra $U_q(\g)$ has a $Q$-grading defined by
$\on{deg}(x_{i,m}^\pm) = \pm \alpha_i$, $\on{deg}(\phi_{i,m}^\pm)
= 0$ for $i\in I$ and $m\in\ZZ$.

By Chari's result \cite[Proposition 1.6]{Cha2}, there is an involutive
automorphism $\hat{\omega} : U_q(\g)\rightarrow U_q(\g)$ defined by
($i\in I$, $m,r\in\Z$, $r\neq 0$)
$$\hat{\omega}(x_{i,m}^\pm) = - x_{i,-m}^\mp,\quad
\hat{\omega}(\phi_{i,\pm m}^\pm) = \phi_{i,\mp m}^{\mp},\quad
\hat{\omega}(h_{i,r}) = - h_{i,-r}.$$
Besides, it satisfies (see the proof of \cite[Proposition 1.6]{Cha2}):
\begin{equation}\label{formomega}\hat{\omega}(e_0) \in \CC^* f_0\text{
    and } \hat{\omega}(e_i) = -f_i\text{ for }i\in I).\end{equation}

Let $U_q(\g)^\pm$ (resp. $U_q(\g)^0$) be the subalgebra of $U_q(\g)$
generated by the $x_{i,r}^\pm$ where $i\in I, r\in\ZZ$ (resp. by the
$\phi_{i,\pm r}^\pm$ where $i\in I$, $r\geq 0$). We have a triangular
decomposition \cite{bec}
\begin{equation}\label{fdecomp}U_q(\g)\simeq U_q(\g)^-\otimes
  U_q(\g)^0 \otimes U_q(\g)^+.\end{equation}

\subsection{Borel algebra}\label{boralg}\label{subsec:Utg}

\begin{defi} The Borel algebra $U_q(\bo)$ is 
the subalgebra of $U_q(\g)$ generated by $e_i$ 
and $k_i^{\pm1}$ with $0\le i\le n$. 
\end{defi}
This is a Hopf subalgebra of $U_q(\g)$. 
The algebra $U_q(\mathfrak{b})$ 
contains 
the Drinfeld generators
$x_{i,m}^+$, $x_{i,r}^-$, $k_i^{\pm 1}$, $\phi_{i,r}^+$ 
where $i\in I$, $m \geq 0$ and $r > 0$. 

Let $U_q(\bo)^\pm = U_q(\g)^\pm\cap U_q(\bo)$ and
$U_q(\bo)^0 = U_q(\g)^0\cap U_q(\bo)$. 
Then we have 
\begin{align*}
U_q(\bo)^+ = \langle x_{i,m}^+\rangle_{i\in I, m\geq 0},\quad
U_q(\bo)^0 = 
\langle\phi_{i,r}^+,k_i^{\pm1}\rangle_{i\in I, r>0}.
\end{align*} 
It follows from \cite{bec} that we have a triangular decomposition
\begin{equation}\label{triandecomp}
U_q(\bo)\simeq U_q(\bo)^-\otimes 
U_q(\bo)^0 \otimes U_q(\bo)^+.
\end{equation}
Denote $\tb \subset U_q(\bo)$ the subalgebra generated by $\{k_i^{\pm 1}\}_{i\in I}$. 

\section{Representations of Borel algebras}\label{catO}

In this section we review results on representations of the Borel
algebra $U_q(\bo)$, in particular on the category $\mathcal{O}$
defined in \cite{HJ} and on finite-dimensional representations of
$U_q(\g)$.

\subsection{Highest $\ell$-weight modules}\label{lhwm}
Set $\tb^*=\bigl(\C^\times\bigr)^I$, and endow it with a group
structure by pointwise multiplication.  We define a group morphism
$\overline{\phantom{u}}:P \longrightarrow \tb^*$ by setting
$\overline{\omega_i}(j)=q_i^{\delta_{i,j}}$. We shall use the standard
partial ordering on $\tb^*$:
\begin{align}
\omega\leq \omega' \quad \text{if $\omega \omega'^{-1}$ 
is a product of $\{\ga_i^{-1}\}_{i\in I}$}.
\label{partial}
\end{align}

For a $U_q(\mathfrak{b})$-module $V$ and $\omega\in \tb^*$, we set
\begin{align}
V_{\omega}=\{v\in V \mid  k_i\, v = \omega(i) v\ (\forall i\in I)\}\,,
\label{wtsp}
\end{align}
and call it the weight space of weight $\omega$. 
For any $i\in I$, $r\in\ZZ$ we have $\phi_{i,r}^\pm (V_\omega)\subset V_\omega$
and $x_{i,r}^\pm (V_{\omega}) \subset V_{\omega \ga_i^{\pm 1}}$.
We say that $V$ is $\tb$-diagonalizable 
if $V=\underset{\omega\in \tb^*}{\bigoplus}V_{\omega}$.

\begin{defi} A series $\Psib=(\Psi_{i, m})_{i\in I, m\geq 0}$ 
of complex numbers such that 
$\Psi_{i,0}\neq 0$ for all $i\in I$ 
is called an $\ell$-weight. 
\end{defi}

We denote by $\tb^*_\ell$ the set of $\ell$-weights. 
Identifying $(\Psi_{i, m})_{m\geq 0}$ with its generating series we shall
write
\begin{align*}
\Psib = (\Psi_i(z))_{i\in I},
\quad
\Psi_i(z) = \underset{m\geq 0}{\sum} \Psi_{i,m} z^m.
\end{align*}
Since each $\Psi_i(z)$ is an invertible formal power series,
$\tb^*_\ell$ has a natural group structure. 
We have a surjective morphism of groups
$\varpi : \tb^*_\ell\rightarrow \tb^*$ given by 
$\varpi(\Psib)(i)=\Psi_{i,0}$.
For a $U_q(\mathfrak{b})$-module $V$ and $\Psib\in\tb_\ell^*$, 
the linear subspace
\begin{align}
V_{\Psibs} =
\{v\in V\mid
\exists p\geq 0, \forall i\in I, 
\forall m\geq 0,  
(\phi_{i,m}^+ - \Psi_{i,m})^pv = 0\}
\label{l-wtsp} 
\end{align}
is called the $\ell$-weight space of $V$ of $\ell$-weight $\Psib$. 

\begin{defi} A $U_q(\mathfrak{b})$-module $V$ is said to be 
of highest $\ell$-weight 
$\Psib\in \tb^*_\ell$ if there is $v\in V$ such that 
$V =U_q(\mathfrak{b})v$ 
and the following hold:
\begin{align*}
e_i\, v=0\quad (i\in I)\,,
\qquad 
\phi_{i,m}^+v=\Psi_{i, m}v\quad (i\in I,\ m\ge 0)\,.
\end{align*}
\end{defi}

The $\ell$-weight $\Psib\in \tb^*_\ell$ is uniquely determined by $V$. 
It is called the highest $\ell$-weight of $V$. 
The vector $v$ is said to be a highest $\ell$-weight vector of $V$.

\begin{prop}\label{simple} \cite{HJ}
For any $\Psib\in \tb^*_\ell$, there exists a simple 
highest $\ell$-weight module $L(\Psib)$ of highest $\ell$-weight
$\Psib$. This module is unique up to isomorphism.
\end{prop}

The submodule of $L(\Psib)\otimes L(\Psib')$ generated by the tensor
product of the highest $\ell$-weight vectors is of highest
$\ell$-weight $\Psib\Psib'$. In particular, $L(\Psib\Psib')$ is a
subquotient of $L(\Psib)\otimes L(\Psib')$.

\begin{defi}\cite{HJ}
For $i\in I$ and $a\in\CC^\times$, let 
\begin{align}
L_{i,a}^\pm = L(\Psib_{i,a})
\quad \text{where}\quad 
(\Psib_{i,a})_j(z) = \begin{cases}
(1 - za)^{\pm 1} & (j=i)\,,\\
1 & (j\neq i)\,.\\
\end{cases} 
\label{fund-rep}
\end{align}
\end{defi}
We call $L_{i,a}^+$ (resp. $L_{i,a}^-$) a positive (resp. negative)
prefundamental representation in the category $\mathcal{O}$.

\begin{example} In the case $\gb = sl_2$, $L_{1,a}^+$ carries a basis
  $L_{1,a}^+=\oplus_{j\geq 0} \CC v_j$ with the explicit action ($r,j
  \geq 0$, $p > 0$, $v_{-1} = 0$):
\begin{align*}
&x_{1,r}^+v_j = \delta_{r,0} v_{j-1}\,,\quad 
x_{1,p}^-v_j = \frac{- aq^{-j}\delta_{p,1}[j+1]_q}{q-q^{-1}} v_{j+1}\,,\quad
\phi_1^+(z)v_j = q^{-2j}(1-za)v_j.
\end{align*}
\end{example}

\begin{defi}\label{oned}\cite{HJ}
For $\omega\in \tb^*$, let 
$$[\omega] = L(\Psib_\omega)
\quad \text{where}\quad 
(\Psib_\omega)_i(z) = \omega(i) \quad (i\in I).$$
\end{defi}
Note that the representation $[\omega]$ is $1$-dimensional with a
trivial action of $e_0,\cdots, e_n$.  It is class a zero
prefundamental representation.  For $\lambda\in P$, we will simply use
the notation $[\lambda]$ for the representation
$[\overline{\lambda}]$.

For $a\in\CC^\times$, the subalgebra $U_q(\mathfrak{b})$ is stable
under $\tau_a$.  Denote its restriction to $U_q(\mathfrak{b})$ by the
same letter.  Then the pullbacks of the $U_q(\mathfrak{b})$-modules
$L_{i,b}^\pm$ by $\tau_a$ is $L_{i,ab}^\pm$.

\subsection{Category $\mathcal{O}$}\label{precato}

For $\lambda\in \tb^*$, we set $D(\lambda )=
\{\omega\in \tb^* \mid \omega\leq\lambda\}$.

\begin{defi}\cite{HJ} A $U_q(\mathfrak{b})$-module $V$ 
is said to be in category $\mathcal{O}$ if:

i) $V$ is $\tb$-diagonalizable,

ii) for all $\omega\in \tb^*$ we have 
$\dim (V_{\omega})<\infty$,

iii) there exist a finite number of elements 
$\lambda_1,\cdots,\lambda_s\in \tb^*$ 
such that the weights of $V$ are in 
$\underset{j=1,\cdots, s}{\bigcup}D(\lambda_j)$.
\end{defi}

The category $\mathcal{O}$ is a monoidal category. 

\newcommand{\mfr}{\mathfrak{r}}
Let 
$\mfr$
be the subgroup of $\tb^*_\ell$
consisting of $\Psib$ such that 
$\Psi_i(z)$ is rational for any $i\in I$.

\begin{thm}\label{class}\cite{HJ} Let $\Psib\in\tb^*_\ell$. 
The simple module $L(\Psib)$ is in category 
$\mathcal{O}$ if and only if $\Psib\in \mfr$. Then it is a subquotient
of a tensor product of (positive, negative, zero) prefundamental representations.
Moreover, 
for $V$ in category $\mathcal{O}$, $V_{\Psib}\neq 0$ implies $\Psib\in\mfr$.
\end{thm}

Let 
$\mathcal{E}_\ell\subset \Z^{\mfr}$
be the ring of maps
$c : \mfr\rightarrow \ZZ$  
satisfying $c(\Psib) = 0$ for all 
$\Psib$ such that $\varpi(\Psib)$ is outside a finite union
of sets of the form $D(\mu)$ and such that for 
each $\omega\in \tb^*$, there are finitely many $\Psib$
such that $\varpi(\Psib) = \omega$ and $c(\Psib)\neq 0$.
Similarly, let $\mathcal{E}\subset \Z^{\tb^*}$ 
be the ring of maps $c : \tb^* \rightarrow \ZZ$ 
satisfying 
$c(\omega) = 0$ for all $\omega$ outside 
a finite union of sets of the form $D(\mu)$. 
The map $\varpi$ is naturally extended to a surjective ring morphism 
$\varpi : \mathcal{E}_\ell\rightarrow \mathcal{E}$. 

For $\Psib\in\mfr$  
(resp. $\omega\in\tb^*$), we define 
$[\Psib] = \delta_{\Psibs,.}\in\mathcal{E}_\ell$ 
(resp. $[\omega] = \delta_{\omega,.}\in\mathcal{E}$).

Let $V$ be a $U_q(\mathfrak{b})$-module in category $\mathcal{O}$. 
We define \cite{Fre, HJ} the $q$-character of $V$ 
\begin{align}
\chi_q(V) = 
\sum_{\Psibs\in\mfr}  
\mathrm{dim}(V_{\Psibs}) [\Psib]\in \mathcal{E}_\ell\,.
\label{qch}
\end{align}

\begin{example}\label{onedim} For $\omega\in\tb^*$, the $q$-character
  of the $1$-dimensional representation
  $[\omega]$ is just its $\ell$-highest weight $\chi_q([\omega]) =
  [\omega]$. That is why the use of the same notation $[\omega]$ will
  not lead to confusion.
\end{example}

Similarly we define the ordinary character of $V$ to be
an element of $\mathcal{E}$
\begin{align}
\chi(V) = \varpi(\chi_q(V)) =  \sum_{\omega\in\tb^*} 
\text{dim}(V_\omega) [\omega]\,.
\label{ch}
\end{align}
For $V$ in category $\mathcal{O}$ which has a unique $\ell$-weight $\Psib$
whose weight is maximal,
we also consider its normalized $q$-character $\tilde{\chi}_q(V)$
and normalized character $\tilde{\chi}(V)$ by
\begin{align*}
\tilde{\chi}_q(V) = [\Psib^{-1}]\cdot\chi_q(V)\,,
\quad 
\tilde{\chi}(V) = \varpi(\tilde{\chi}_q(V))\,.
\end{align*}
Let $\text{Rep}(U_q(\mathfrak{b}))$ be the Grothendieck ring of the
category $\mathcal{O}$.  
\begin{prop} The $q$-character morphism 
$$
\chi_q : \text{Rep}(U_q(\mathfrak{b}))\rightarrow
\mathcal{E}_\ell,\quad [V]\mapsto \chi_q(V),
$$
is an injective ring morphism.
\end{prop}

\subsection{Finite-dimensional representations}\label{fdrep}

Let $\mathcal{C}$ be the category of (type $1$) finite-dimen\-sional
representations of $U_q(\Glie)$.

For $i\in I$, let $P_i(z)\in\mathbb{C}[z]$ be a polynomial with constant 
term $1$. Set
\begin{align*}
\Psib = (\Psi_i(z))_{i\in I},
\quad
\Psi_i(z) = 
q_i^{\on{deg}(P_i)}\frac{P_i(zq_i^{-1})}{P_i(zq_i)}.
\end{align*}
Then $L(\Psib)$ is finite-dimensional.  Moreover the action of
$U_q(\mathfrak{b})$ can be uniquely extended to an action of the full
quantum affine algebra $U_q(\Glie)$, and any simple object in the
category $\mathcal{C}$ is of this form.

\begin{rem} Let $L(\Psib')$ be any finite-dimensional module in the
  category $\mathcal{O}$. We claim that there is $\Psib$ as above and
  $\omega\in\tb^*$ such that $L(\Psib') \simeq L(\Psib)\otimes
  [\omega]$. Since $[\omega]$ is just a one-dimensional
  representation, this means that, up to a slight twisting of the
  action of the Cartan elements, $L(\Psib')$ is a representation of
  $U_q(\Glie)$.  This statement is known in the case $\gb = sl_2$
  \cite{bt}. To prove it in general, it suffices to prove that $\Psib'
  = \Psib \Psib_{\omega}$. This is clear by $sl_2$-reduction as, for
  each $i\in I$, the subalgebra of $U_q(\bo)$ generated by the
  $k_i^{\pm 1}$, $x_{i,m}^+$, $x_{i,m+1}^-$, $\phi_{i,m}^+$, $m \geq
  0$ is isomorphic to the Borel algebra of $U_{q_i}(\widehat{sl}_2)$.
\end{rem}

Following \cite{Fre}, consider the ring of Laurent polynomials $\Yim =
\ZZ[Y_{i,a}^{\pm 1}]_{i\in I,a\in\CC^*}$ in the indeterminates
$\{Y_{i,a}\}_{i\in I, a\in \C^*}$.  Let $\mathcal{M}$ be the group of
monomials of $\Yim$. For example, for $i\in I, a\in\CC^*$, define
$A_{i,a}\in\mathcal{M}$ to be
$$
Y_{i,aq_i^{-1}}Y_{i,aq_i}
\Bigl(\prod_{\{j\in I|C_{j,i} = -1\}}Y_{j,a}
\prod_{\{j\in I|C_{j,i} = -2\}}Y_{j,aq^{-1}}Y_{j,aq}
\prod_{\{j\in I|C_{j,i} =
-3\}}Y_{j,aq^{-2}}Y_{j,a}Y_{j,aq^2}\Bigr)^{-1}\,.
$$
For a monomial 
$m = \prod_{i\in I, a\in\CC^*}Y_{i,a}^{u_{i,a}}$, 
we consider its `evaluation on $\phi^+(z)$'. 
By definition it is an element 
$m(\phi(z))\in\mfr$  
given by
$$m\bigl(\phi(z))=
\prod_{i\in I, a\in\CC^*}
\left(Y_{i,a}(\phi(z))\right)^{u_{i,a}}\text{ where }
\Bigl(Y_{i,a}\bigl(\phi(z)\bigr)\Bigr)_j
=\begin{cases}
\displaystyle{q_i\frac{1-a q_i^{-1}z}{1-aq_iz}}& (j=i),\\
1 & (j\neq i).\\
\end{cases}$$
This defines an injective group morphism 
$\mathcal{M}\rightarrow \mfr$. 
We identify a monomial $m\in\mathcal{M}$ with its image in 
$\mfr$. 
Note that $\varpi(Y_{i,a}) = \overline{\omega_i}$.

It is proved in \cite{Fre} that a finite-dimensional
$U_q(\Glie)$-module $V$ satisfies 
$V = \bigoplus_{m\in\mathcal{M}} V_{m\left(\phi(z)\right)}$.
In particular, $\chi_q(V)$ can be viewed  
as an element of $\Yim$.

A monomial $M\in\mathcal{M}$ is said to be dominant if
$M\in\ZZ[Y_{i,a}]_{i\in I, a\in\CC^*}$.  For $L(\Psib)$ a
finite-dimensional simple $U_q(\Glie)$-module, $\Psib =
M\bigl(\phi(z)\bigr)$ holds for some dominant monomial
$M\in\mathcal{M}$. This representation will be denoted by $L(M)$.

For example, for $i\in I$, $a\in\CC^*$ and $k\geq 0$, we have the
Kirillov-Reshetikhin (KR) module
\begin{align}
W_{k,a}^{(i)} = L(Y_{i,a}Y_{i,aq_i^2}\cdots Y_{i,aq_i^{2(k-1)}})\,.
\label{KRmod}
\end{align}
The representations $W_{1,a}^{(i)} = L(Y_{i,a})$ are called
fundamental representations.
\begin{example}\label{exkr} In the case $\gb = sl_2$, we have ($k\geq
  0$, $a\in\CC^*$) \cite{Fre}:
$$\chi_q(W_{k,aq^{1-2k}}^{(1)}) = Y_{aq^{-1}}Y_{aq^{-3}}\cdots Y_{aq^{-2k+1}}
(1 + A_{1,a}^{-1} + A_{1,a}^{-1}A_{1,aq^{-2}}^{-1} 
+ \cdots + A_{1,a}^{-1} \cdots A_{1,aq^{-2(k - 1)}}^{-1})\,,
$$
and $W_{k,aq^{1-2k}}^{(1)}$ carries a basis $(w_0,\cdots,w_k)$ 
with the explicit action ($r\in\ZZ$, $0\leq j\leq k$, $w_{-1} = w_{k+1} = 0$): 
\begin{align*}
&x_{1,r}^+w_j=a^rq^{2r(-j+1)} w_{j-1}\,,
\quad 
x_{1,r}^-w_j=a^rq^{-2rj}[j+1]_q[k - j]_qw_{j+1}\,,
\\
&\phi^{\pm}_1(z)w_j=q^{k-2j}
\frac{(1-q^{-2k}za)(1-q^{2}za)}{(1-q^{-2j+2}za)(1-q^{-2j}za)}w_j\,.
\end{align*}
\end{example}

For $m$ a dominant monomial, we will denote $\tilde{L}(m) = L(m
(\varpi(m))^{-1})$.

\subsection{The dual category $\mathcal{O}^*$}\label{dualcat}

For $V$ a $\tb$-diagonalizable $U_q(\bo)$-module, 
we define a structure of $U_q(\bo)$-module on its
graded dual $V^* = \oplus_{\beta\in \tb^*} V_\beta^*$ by
\begin{align*}
(x\,u)(v)=u\bigl(S^{-1}(x)v\bigr)\quad
(u\in  V^*, \ v\in V,\ x \in U_q(\mathfrak{b})).
\end{align*}

\begin{defi} Let $\mathcal{O}^*$ be the category of
  $\tb$-diagonalizable $U_q(\bo)$-modules $V$ such that $V^*$ is in
  category $\mathcal{O}$.
\end{defi}

A $U_q(\mathfrak{b})$-module $V$ is said to be of lowest $\ell$-weight 
$\Psib\in \tb^*_\ell$ if there is $v\in V$ such that $V =U_q(\mathfrak{b})v$ 
and the following hold:
\begin{align*}
U_q(\bo)^- v = \CC v\,,
\qquad 
\phi_{i,m}^+v=\Psi_{i, m}v\quad (i\in I,\ m\ge 0)\,.
\end{align*}
For $\Psib\in\tb^*_\ell$, we have the simple $U_q(\bo)$-module
$L'(\Psib)$ of lowest $\ell$-weight $\Psi$.  We have the notion of
characters and $q$-characters for category $\mathcal{O}^*$ as in
Section \ref{precato}.

\begin{prop}\label{dualweight}\cite{HJ} For $\Psib\in \tb^*_\ell$ we
  have $(L'(\Psib))^* \simeq L(\Psib^{-1})$.
\end{prop}
We will consider the prefundamental representations $R_{i,a}^\pm$ in
$\mathcal{O}^*$ defined by $(R_{i,a}^\pm)^* \simeq L_{i,a}^\mp$.

\begin{example} In the case $\gb = sl_2$, $R_{1,a}^+$ carries a basis
  $R_{1,a}^+=\oplus_{j\geq 0} \CC v_j^*$ with the explicit action
  ($r,j \geq 0$, $p > 0$, $v_{-1}^* = 0$):
\begin{align*}
&x_{1,r}^+v_j^* = \delta_{r,0} q^{2j} v_{j+1}^*\,,\quad
 x_{1,p}^-v_j^* = \frac{- aq^{1-j}\delta_{p,1}[j]_q}{q-q^{-1}} v_{j-1}^*\,,\quad
\phi_1^+(z)v_j^* = q^{2j}(1-za)v_j^*.
\end{align*}
\end{example}

\subsection{The opposite Borel and the category $\overline{\mathcal{O}}$}
It will be also convenient to use the opposite Borel $U_q(\bo^-) =
\hat{\omega}(U_q(\bo))$. By (\ref{formomega}), $U_q(\bo^-)$ is the
subalgebra of $U_q(\g)$ generated by $f_i$ and $k_i^{\pm 1}$ with
$0\le i\le n$. Hence it is a Hopf subalgebra of $U_q(\g)$.  Let
further $U_q(\bo^-)^\pm = U_q(\g)^\pm\cap U_q(\bo^-)$ and
$U_q(\bo^-)^0 = U_q(\g)^0\cap U_q(\bo)$.  Then we have
\begin{equation}\label{negger}
U_q(\bo^-)^- = \langle x_{i,-m}^-\rangle_{i\in I, m\geq 0},\quad
U_q(\bo^-)^0 = 
\langle\phi_{i,-r}^-,k_i^{\pm 1}\rangle_{i\in I, r>0}.
\end{equation} 
We have a triangular decomposition
\begin{equation}\label{triandecomp2}
U_q(\bo^-)\simeq U_q(\bo^-)^-\otimes 
U_q(\bo^-)^0 \otimes U_q(\bo^-)^+.
\end{equation}
By mimicking the definition of the category $\mathcal{O}$, we can
define the category $\overline{\mathcal{O}}$ of $U_q(\bo^-)$-modules.
For $V$ a $U_q(\bo)$-module, we have a structure of
$U_q(\bo^-)$-module on $V$ denoted by $V^{\hat{\omega}}$ and defined
by twisting the action by the automorphism $\hat{\omega}$.  In
particular, we get the simple objects $\overline{L}(\Psib) =
(L'(\Psib))^{\hat{\omega}}$ of the category $\overline{\mathcal{O}}$.
Hence, we have a parametrization of simple objects, as well as
$q$-character theory, in the category $\overline{\mathcal{O}}$ as for
the category $\mathcal{O}$. In particular we have the prefundamental
representations $\overline{L}_{i,a}^\pm =
(R_{i,a^{-1}}^\pm)^{\hat{\omega}}$ in the category
$\overline{\mathcal{O}}$.

\begin{example} In the case $\gb = sl_2$, $\overline{L}_{1,a}^+$ carries a
  basis $\overline{L}_{1,a}^+=\oplus_{j\geq 0} \CC v_j^*$ with the
  explicit action ($r,j \geq 0$, $p > 0$, $v_{-1}^* = 0$):
\begin{align*}
&x_{1,-r}^-v_j^* = -\delta_{r,0} q^{2j} v_{j+1}^*\,,\quad
 x_{1,-p}^+v_j^* = \frac{a^{-1}q^{1-j}\delta_{p,1}[j]_q}{q-q^{-1}} v_{j-1}^*\,,\quad
\phi_1^-(z)v_j^* = q^{2j}(1-(za)^{-1})v_j^*.
\end{align*}
\end{example}

\section{Baxter's relations in category $\mathcal{O}$}\label{bacato}

In this section we prove a uniform explicit $q$-character formula for
positive prefundamental representations (Theorem \ref{formuachar}): it
is equal to the product of the highest $\ell$-weight and the ordinary
character (which does not depend on the spectral parameter). We prove
that this implies for each finite-dimensional representation $V$ of
$U_q(\g)$ the existence of a relation in the Grothendieck ring of
$\mathcal{O}$ obtained from the $q$-character of $V$ (Theorem
\ref{relations}). This is our first main result, which is a
generalization of Baxter's TQ relations discussed in the Introduction.
We also prove that an arbitrary tensor product of positive
prefundamental representations in simple (Theorem \ref{stensor}).

\subsection{$q$-characters of positive prefundamental representations}

\begin{thm}\label{formuachar} Let $i\in I$. Then we have for any $a\in\CC^*$,
$$\chi_q(L_{i,a}^+) = \Psib_{i,a}\times \chi(L_{i,a}^+)\text{ and }\chi_q(R_{i,a}^+) = \Psib_{i,a}\times \chi(R_{i,a}^+).$$
\end{thm}

\begin{rem}
  (i) The characters $\chi(L_{i,a}^+) = \chi(L_{i,a}^-)$, are
  explicitly known and are equal to each other \cite[Theorem 6.4]{HJ}.
  Since $(R_{i,a}^+)^*\simeq L_{i,a}^-$, the character $\chi(R_{i,a}^+)$
  is also explicitly known. Besides, the formula is uniform. Hence, as
  the highest $\ell$-weight $\Psib_{i,a}$ is known, the statement of
  Theorem \ref{formuachar} is an explicit uniform $q$-character
  formula.

  (ii) Theorem \ref{formuachar} implies that the normalized
  $q$-characters $\tilde{\chi}_q(L_{i,a}^+) = \chi(L_{i,a}^+)$,
  $\tilde{\chi}_q(R_{i,a}^+) = \chi(R_{i,a}^+)$ do not depend on the
  spectral parameter $a$.

  (iii) When the multiplicity $N_i$ of $\alpha_i$ in the maximal root
  of $\gb$ is equal to $1$, this result was established in
  \cite{HJ}. It relies on an asymptotic construction of the
  representation $L_{i,a}^+$ which is only valid if $N_i = 1$. Our
  proof is different and works for all cases.
  \end{rem}

\begin{example} In the case $\gb = sl_2$, we have
$$\chi_q(L_{1,a}^+) = [(1 - za)]\left(\sum_{r\geq 0}
  [\overline{-2r\omega_1}]\right)\text{ , }\chi_q(R_{1,a}^+) = [(1 -
za)]\left(\sum_{r\geq 0} [\overline{2r\omega_1}]\right).$$
\end{example}

\begin{rem} Although positive and negative prefundamental representations
  have the same character, their $q$-characters are very
  different. For instance, in the case $\gb = sl_2$, we have
  $$\chi_q(L_{1,a}^-) = [(1 - za)^{-1}]\left(\sum_{r\geq 0}
  (A_{1,a}A_{1,aq^{-2}}\cdots A_{1,aq^{-2(r-1)}})^{-1}\right),$$ 
  $$\chi_q(R_{1,a}^-) = [(1 -
za)^{-1}]\left(\sum_{r\geq 0} (A_{1,a}A_{1,aq^2}\cdots A_{1,aq^{2(r-1)}})\right).$$
The reader may also look at the geometric $q$-character formulas
  for negative prefundamental representations established in \cite{HL}
  (see \cite[Remark 4.19]{HL} for details).
\end{rem}

\subsection{Proof of Theorem \ref{formuachar}} We will use the
following technical result.

\begin{lem}\label{kr} Let $i\in I$, $a\in\CC^*$, $0\leq K\leq k$. Let $m$ be a monomial
  occurring in $\tilde{\chi}_q(W_{k,aq_i^{1-2k}}^{(i)})$ such that the
  multiplicity of $-\alpha_i$ in $\varpi(m)$ is lower than $K$. Then
  $m$ is a monomial of $\tilde{\chi}_q(W_{K,aq_i^{1-2K}}^{(i)})$.
\end{lem}

\begin{proof} Let us prove the result by induction on $k-K\geq 0$. It
  it is trivial if $k = K$. 
  Now, suppose in general that $k > K$. We have \cite[Lemma 5.8]{hcr}
\begin{equation}\label{recu}\tilde{\chi}_q(W_{k,aq_i^{1-2k}}^{(i)})
  \in \tilde{\chi}_q(W_{k-1,aq_i^{3-2k}}^{(i)}) +
  (A_{i,a}A_{i,aq_i^{-2}}\cdots A_{i,aq_i^{2-2k}})^{-1}
  \ZZ[A_{j,b}^{-1}]_{j\in I, b\in\CC^*}.\end{equation}
Hence $m$ is a monomial of $\tilde{\chi}_q(W_{k-1,aq_i^{3-2k
  }}^{(i)})$. We can conclude with the induction
hypothesis.\end{proof}

Now we complete the proof of the Theorem.

\begin{proof} Let us explain the proof for the $L_{i,a}^+$. The same
  proof gives the analogous result for the $\overline{L}_{i,a}^+$. By
  using
  $\hat{\omega}$ this implies the result for the $R_{i,a}^+$. For
  $k\geq 0$, $L_{i,aq_i^{-2k}}^+$ is a subquotient of
$$L_{i,a}^+\otimes \tilde{L}(Y_{i,aq_i^{-1}}Y_{i,aq_i^{-3}}\cdots Y_{i,aq_i^{-2k + 1}}).$$
Suppose that an $\ell$-weight $\Psib = (\Psi_i(z))_{i\in I}$ occurring in
$\tilde{\chi}_q(L_{i,a}^+)$ has a pole or a zero $b\in\CC^*$. Since
$(\Psib)_{q_i^{-2k}} = (\Psi_i(zq_i^{-2k}))_{i\in I}$ occurs in
$\tilde{\chi}_q(L_{i,aq_i^{-2k}}^+)$, it can be factorized into
$$(\Psib)_{q_i^{-2k}} = \Psib_k' m_k$$ 
where $\Psib_k'$ (resp. $m_k$) occurs in $\tilde{\chi}_q(L_{i,a}^+)$
(resp. in $\tilde{\chi}_q(W_{k,aq_i^{1-2k}})$).  Let $K$ be the
multiplicity of $- \alpha_i$ in $\varpi(\Psib)\in -Q^+$. By Lemma \ref{kr},
the monomial $m_k$ occurs in $\tilde{\chi}_q(W_{K,aq_i^{1-2K}})$. We have
proved that $(\Psib)_{q_i^{-2k}}$ occurs in
$$\tilde{\chi}_q(L_{i,a}^+)\tilde{\chi}_q(\tilde{L}(Y_{i,aq_i^{-1}}
Y_{i,aq_i^{-3}}\cdots Y_{i,aq_i^{1-2K}})).$$ 
In this product, there is only a finite number
of terms of weight $\varpi(\Psib)$. But for each $k\geq K$, one of
this term has a pole or a zero $b q_i^{-2k}$. Contradiction. So
$$\tilde{\chi}_q(L_{i,a}^+) = \chi(L_{i,a}^+) = \chi(L_{i,1}^+).$$
\end{proof}

\subsection{Baxter's relations} Now we have the following.

\begin{cor}\label{quotient} For $i\in I$ and $a\in\CC^*$, we have 
  $$[\overline{\omega_i}]\frac{\chi_q(L_{i,aq_i^{-1}}^+)}{\chi_q(L_{i,aq_i}^+)}
  =
  [\overline{\omega_i}]\frac{\chi_q(R_{i,aq_i^{-1}}^+)}{\chi_q(R_{i,aq_i}^+)}
  = Y_{i,a}.
$$
\end{cor}

\begin{proof} First, by definition of $Y_{i,a}$, we have the relation
  for highest $\ell$-weights:
$$[\overline{\omega_i}]\frac{\Psib_{i,aq_i^{-1}}}{\Psib_{i,aq_i}} = Y_{i,a}.$$
Now the character of prefundamental representations do not depend on
the spectral parameter:
$$\chi(L_{i,aq_i^{-1}}^+) = \chi(L_{i,aq_i}^+) \text{ and
}\chi(R_{i,aq_i^{-1}}^+) = \chi(R_{i,aq_i}^+).$$
Hence Theorem \ref{formuachar} implies the result.
\end{proof}

\begin{rem} The formulas we obtain in Corollary \ref{quotient} can
  also be seen as a change of variables (analogous to those used in
  \cite[Section 5.2.2]{HL}).\end{rem}

We can now prove generalized Baxter's relations in the category
$\mathcal{O}$ (and $\mathcal{O}^*$).

\begin{thm}\label{relations} Let $V$ be a finite-dimensional representation of
  $U_q(\g)$. Replace in $\chi_q(V)$ each variable $Y_{i,a}$ by
  $[\omega_i]\frac{[L_{i,aq_i^{-1}}^+]}{[L_{i,aq_i}^+]}$ and
  $\chi_q(V)$ by $[V]$. Then, multiplying by denominators, we get a
  relation in the Grothendieck ring of $\mathcal{O}$.

  Similarly, replacing $Y_{i,a}$ by
  $[\omega_i]\frac{[R_{i,aq_i^{-1}}^+]}{[R_{i,aq_i}^+]}$, we get a
  relation in the Grothendieck ring of $\mathcal{O}^*$.
\end{thm}

\begin{proof} Since the $q$-character morphism is injective, the result
  follows from Corollary \ref{quotient}.  We have also used the
  $q$-character formula $\chi_q([\omega_i]) = [\overline{\omega_i}]$
  for the $1$-dimensional representation $[\omega_i]$, as explained in
  Example \ref{onedim}.
\end{proof}

\begin{example} (i) Our result generalizes the following known example in
  the case $\gb = sl_2$. We have for $a\in\CC^*$ a relation in the
  Grothendieck ring of $\mathcal{O}$:
$$[L(Y_{1,a})][L_{1,aq}^+] = [L_{1,aq^{-1}}^+][\omega_1] +
[L_{1,aq^3}^+][-\omega_1].$$
Similarly, we have in the Grothendieck ring of the category $\mathcal{O}^*$: 
$$[L(Y_{1,a})][R_{1,aq}^+] = [R_{1,aq^{-1}}^+][\omega_1] +
[R_{1,aq^3}^+][-\omega_1].$$ 

(ii) For $\gb = sl_3(\CC)$, we have 
$$\chi_q(L(Y_{1,1})) = Y_{1,1} + Y_{1,q^2}^{-1} Y_{2,q} + Y_{2,q^3}^{-1}.$$
Hence we have the following Baxter relation in the Grothendieck ring
of $\mathcal{O}$:
$$[L(Y_{1,1})][L_{1,q}^+][L_{2,q^2}^+]
= [L_{1,q^{-1}}^+][L_{2,q^2}^+][\omega_1] 
+ [L_{1,q^3}^+][L_{2,1}^+][\omega_2 - \omega_1]
+ [L_{1,q}^+][L_{2,q^4}^+][- \omega_1]
.$$

(iii) Let us give another example for $\gb$ of type $B_2$. The
  $q$-character of the $4$-dimensional fundamental representation
  $L(Y_{2,1})$ is
$$\chi_q(L(Y_{2,1})) = Y_{2,1} + Y_{2,q^2}^{-1} Y_{1,q} +
Y_{1,q^5}^{-1}Y_{2,q^4} + Y_{2,q^6}^{-1}.$$
Hence we have the following Baxter relation in the Grothendieck ring
of $\mathcal{O}$:
$$[L(Y_{2,1})][L_{2,q}^+][L_{1,q^3}^+][L_{2,q^5}^+]
= [L_{2,q^{-1}}^+][L_{1,q^3}^+][L_{2,q^5}^+][\omega_2] 
+ [L_{2,q^3}^+][L_{1,q^{-1}}^+][L_{2,q^5}^+][\omega_1 - \omega_2]$$
$$+ [L_{2,q^3}^+][L_{1,q^7}^+][L_{2,q}^+][\omega_2 - \omega_1]
+ [L_{2,q^7}^+][L_{2,q}^+][L_{1,q^3}^+][- \omega_2].$$
\end{example}

\begin{rem}
  (i) By taking the duals, Baxter's relations in Theorem
  \ref{relations} may also be written in terms of negative
  prefundamental representations. For example, in the case $\gb = sl_2$, for
  $a\in\CC^*$ we get in the Grothendieck ring of $\mathcal{O}$:
$$[L(Y_{1,aq^2})][L_{1,aq}^-] = [L_{1,aq^{-1}}^-][-\omega_1] + [L_{1,aq^3}^-][\omega_1],$$
and in the Grothendieck ring of the category $\mathcal{O}^*$: 
$$[L(Y_{1,aq^{-2}})][R_{1,aq}^-] = [R_{1,aq^{-1}}^-][-\omega_1] + [R_{1,aq^3}^-][\omega_1].$$

(ii) For quantum affine algebras of classical types, some
conjectural relations in the Grothendieck ring of the category
$\mathcal{O}$ have been proposed in \cite{S}. It is not clear to us
whether there is a connection between them and the generalized Baxter
relations that we establish in this paper.
\end{rem}

\subsection{Baxter's relations as tensor product decomposition}
In this subsection we give an additional interpretation of Baxter's
relations of Theorem \ref{relations}. This is only included for
completeness of the paper as the results of this subsection will not
be used in the other sections.

Let us first prove the following additional application of Theorem
\ref{formuachar}.

\begin{thm}\label{stensor} An arbitrary tensor product of positive
  (resp. negative) prefundamental representations in the category
  $\mathcal{O}$ is simple. The same holds in the category
  $\mathcal{O}^*$.
\end{thm}

\begin{proof} First let us prove the result for a tensor product $T$
  of negative prefundamental representations in the category
  $\mathcal{O}$. It can be written in form
$$T = \bigotimes_{a\in(\CC^*/q^\ZZ)} \bigotimes_{i\in I, r\in\ZZ} (L_{i,aq^r}^-)^{\otimes n_{i,aq^r}}$$
where we have chosen a representative $a\in\CC^*$ for each class in
$\CC^*/q^\ZZ$ so that for any $i\in I$, $r\leq d_i$, we have
$n_{i,aq^r} = 0$. The ordering in the tensor product is not relevant
for this proof as the Grothendieck ring of the category $\mathcal{O}$
is commutative.  Let $\Psib$ be the highest $\ell$-weight of $T$.  We
prove that $L(\Psib)$ is isomorphic to $T$. First $L(\Psib)$ is a
subquotient of $T$.  So it suffices to prove that the dimensions of
weight spaces of $L(\Psib)$ are greater than those of $T$. 

For $R\leq 0$,
consider the simple module $L_R = \tilde{L}(M_R)$ where the monomial
$M_R$ is defined by
$$M_R = \prod_{a\in(\CC^*/q^\ZZ)} \prod_{i\in I, r > 0} 
\left(\prod_{\{r'\in r + 2 d_i\ZZ|r\geq r' >
    R\}}Y_{i,aq^{r'}}^{n_{i,aq^{r + d_i}}}\right).$$ Then $L_R$ is
isomorphic to a tensor product of (normalized) Kirillov-Reshetikhin
modules
$$L_R\simeq \bigotimes_{a\in(\CC^*/q^\ZZ)} \bigotimes_{i\in I, r\in\ZZ} 
\left(\tilde{L}\left(\prod_{\{r'\in r + 2 d_i\ZZ|r\geq r' >
      R\}}Y_{i,aq^{r'}}\right)\right)^{\otimes n_{i,aq^r}}.$$ Indeed
it suffices to prove the irreducibility of the tensor product when we
replace spectral parameters by their inverse (see \cite[Proposition
4.13]{h3}). Then the result becomes clear as the $q$-character of the
tensor product has a unique dominant monomial (see for example
\cite[Proposition 5.3]{hcr}).

Now, by \cite[Theorem 6.1]{HJ}, the character of $T$ is the limit (as a
formal power series in the negative simple roots) of the character of
$L_R$ when $R\rightarrow -\infty$. So it suffices to prove that the
dimension of weight spaces of $L_R$ are lower than those of $L(\Psib)$.

Consider the tensor product
$$T' = L(\Psib)\otimes \left(\bigotimes_{a\in(\CC^*/q^\ZZ)} \bigotimes_{i\in I,
    r\in\ZZ} (L_{i,aq^{R_{a,i,r}}}^+)^{\otimes n_{i,aq^{r}}}\right)$$
where $R_{a,i,r}$ is the lowest integer $r'$ such that $r'\in r +
2 d_i\ZZ$ and $r' > R - d_i$. Since $T'$ and $L_R$ have the same 
highest $\ell$-weight, $L_R$ is a subquotient of $T'$.

Recall that the $\ell$-weights of $L_R$ are the product of the highest
$\ell$-weight $M_R(\varpi(M_R))^{-1}$ multiplied by a product of
$A_{j,b}^{-1}$, $j\in i$, $b\in\CC^*$ \cite[Theorem 4.1]{Fre2}.
Hence, by Theorem \ref{formuachar}, an $\ell$-weight of $T'$ is an
$\ell$-weight of $L_R$ only if it of the form
$$\Psib'(M_R (\varpi(M_R))^{-1}\Psib^{-1})$$
where $\Psib'$ is an $\ell$-weight of $L(\Psib)$ and $(M_R
(\varpi(M_R))^{-1}\Psib^{-1})$ is the highest $\ell$-weight of the
remaining tensor product of positive prefundamental
representations. We get the result.

The same proof gives the result for negative prefundamental
representations in the category $\mathcal{O}^*$.  By duality, we get
the result for prefundamental representations in the category
$\mathcal{O}$ as well as in the category $\mathcal{O}^*$.
\end{proof}

The tensor product of a simple representation in $\mathcal{O}$
(resp. in $\mathcal{O}^*$) by a $1$-dimensional representation
$[\omega]$, $\omega\in\tb^*$, is clearly simple.  So, multiplying
Baxter's relation of Theorem \ref{relations} by the denominators and
using Theorem \ref{stensor}, we get the following.

\begin{cor} Baxter's relation of Theorem \ref{relations} may be
  interpreted as the decomposition, in the Grothendieck ring of
  ${\mathcal O}$ (resp. of $\mathcal{O}^*$), of the class of the
  tensor product of two simple representations into a sum of classes
  of simple representations.
\end{cor}

\section{Transfer-matrices and polynomiality}\label{trpol}

In this section we state our second main result: polynomiality of the
twisted transfer-matrices (and of their eigenvalues) associated to the
prefundamental representations (Theorem \ref{catOpol}). Our main
application is the proof of a version of the conjecture of Reshetikhin
and the first author (Theorem \ref{fdcase}).  We use this result to
write down the system of Bethe Ansatz equations explicitly in Section
\ref{bae}. We also derive the polynomiality of Drinfeld's Cartan
elements on finite-dimensional representations (Theorem \ref{jordan})
and prove commutativity of the twisted transfer-matrices associated to
representations in the category $\mathcal{O}$ (Theorem \ref{com}).

\subsection{Universal $R$-matrix, $L$-operators and transfer-matrices}

The universal $R$-matrix $\mathcal{R}$ of $U_q(\Glie)$ belongs to the
tensor product $U_q(\Glie)\hat{\otimes} U_q(\Glie)$ (completed for the
$\ZZ$-grading of $U_q(\Glie)$). The Cartan subalgebra of $U_q(\Glie)$
is also slightly completed: elements $e^{h\omega}$, $\omega\in
P\otimes_\ZZ \CC$, are added so that $k_i = e^{h\alpha_i}$ for $i\in
I$, as in \cite{CP} (see also \cite{da}).  We will use analogous
completions of Borel subalgebras. The completed algebras still act on
the representations we consider.

The universal $R$-matrix satisfies the Yang-Baxter equation
\begin{equation}\label{yb}\mathcal{R}_{12}\mathcal{R}_{13}\mathcal{R}_{23}
  = \mathcal{R}_{23}\mathcal{R}_{13}\mathcal{R}_{12}.\end{equation}
In fact, it is known that $\mathcal{R}$ belongs to
$U_q(\mathfrak{b})\hat{\otimes} U_q(\mathfrak{b}^-)$.  Hence for $V$ a
$U_q(\mathfrak{b})$-module, we may define the $L$-operator associated
to $V$
$$L_V(z) = (\pi_{V(z)}\otimes \text{Id}) (\mathcal{R}) \in
End(V)[[z]]\hat{\otimes} U_q(\mathfrak{b}^-),$$ where $\pi_V(z):
U_q(\mathfrak{b})\rightarrow End(V)[[z]]$ is the representation
morphism of $V$ with the $\ZZ$-grading of $U_q(\mathfrak{b})$.

Let $V$ in the category $\mathcal{O}^*$ whose weights are in
$\overline{P}\subset \tb^*$. For $g\in U_q(\mathfrak{b})$ (or in
$\text{End}(V)$), the twisted trace of $g$ on $V$ is
$$Tr_{V,u}(g) = \sum_{\lambda \in \tb^*}
Tr_{V_\lambda}(\pi_V(g))\left(\prod_{i\in I}u_i^{\lambda_i}\right) \in
\CC[[u_i^{\pm 1}]]_{i\in I},$$ where $\lambda_i\in\ZZ$ is defined by
$\lambda(i) = q_i^{\lambda_i}$.

\begin{defi} \label{gr tr}
The twisted transfer-matrix associated to $V$ is
$$t_V(z,u) = (Tr_{V,u}\otimes \text{Id}) (L_V(z))\in
U_q(\mathfrak{b}^-)[[z, u_i^{\pm 1}]]_{i\in I}.$$
\end{defi}

More precisely we have
$$t_V(z,u) \in U_q(\mathfrak{b}^-)[u_i^{\pm 1}]_{i\in I}[[z,v_i]]_{i\in I}$$
where for $i\in I$
$$v_i = \prod_{j\in I} u_j^{C_{j,i}}$$ 
corresponds to a simple root. Hence the product
of twisted transfer-matrices is well-defined.

Note that for $V = R_{i,a}^+$, we have $t_V(z,u)\in
U_q(\mathfrak{b}^-)[[z,v_i]]_{i\in I}$.

\begin{example}\label{extr} Let $i\in I$. By using Section
  (\ref{fact}), we get 
$$t_{L(\overline{\alpha_i})} = v_i k_i^{-1}\text{ and
}t_{L(\overline{\omega_i})}(z,u) = u_i \tilde{k}_i^{-1}$$
where the $\tilde{k}_i = \prod_{j\in I}k_j^{(C^{-1})_{j,i}}$ are
characterized by the relations
$$k_i = \prod_{j\in I}\tilde{k}_j^{C_{j,i}}.$$
\end{example}
For $V$ in $\mathcal{C}$, $t_V(z,u)$ is a finite sum and the variables
$u_i$ can be specialized to any non zero complex values. For example, 
consider the specialization
$$u_i = q_i^{2\sum_{j\in I}(C^{-1})_{j,i}} = q^{2\sum_{j\in I}(DC^{-1})_{i,j}}.$$ 
This means $v_i = q_i^2$ as $DC^{-1}$ is symmetric (and $u = q$ if $\gb = sl_2$). Then we recover the standard non-twisted transfer-matrix
$$t_V(z) = (Tr_V\otimes \text{Id}) (((\prod_{1\leq i\leq n}\tilde{k}_i^2)\otimes 1) L_V(z))\in
U_q(\mathfrak{b}^-)[[z]].$$

\subsection{Commutativity of twisted transfer-matrices}
It can be proved as in \cite[Lemma 2]{Fre} that for $V,V'$ in the
category $\mathcal{O}^*$ whose weights are in $\overline{P}\subset
\tb^*$, and for $W$ an extension of $V$ and $V'$ in the category
$\mathcal{O}^*$, we have
\begin{equation}\label{alg}t_{W}(z,u) = t_V(z,u) + t_{V'}(z,u)\text{
    and }t_{V\otimes V'}(z,u) = t_V(z,u) t_{V'}(z,u).\end{equation}
Let us prove the following stronger result.

\begin{thm}\label{com} For $V, V'$ in the category $\mathcal{O}^*$
  whose weights are in $\overline{P}\subset \tb^*$, we have
$$t_V(z,u)t_{V'}(w,u) = t_{V'}(w,u) t_V(z,u).$$\end{thm}

\begin{rem}
  (i) The result is well-known when $V, V'$ are finite-dimensional in
  the category $\mathcal{C}$ (see \cite[Lemma 2]{Fre}).  But in
  general the action of $U_q(\bo)$ on $V,V'$ can not be extended to an
  action of $U_q(\Glie)$ and we can not evaluate $\pi_{V(z)}\otimes
  \pi_{V'(w)}$ on the universal $R$-matrix $\mathcal{R}$. That is why
  we can not use directly the standard proof for the general
  representations in $\mathcal{O}^*$ we consider.

(ii) The same proof as below gives the result for representations in
the category $\mathcal{O}$.
\end{rem}

\begin{proof} From (\ref{alg}), $t_V(z)$ depends only of the class of
  $V$ in the Grothendieck ring of $\mathcal{O}^*$.  But by
  \cite[Remark 3.13]{HJ}, this ring is commutative. Hence
$$t_V(z,u)t_{V'}(z,u) = t_{V'}(z,u) t_V(z,u).$$
Now let $a\in\CC^*$ and $V_a$ be the $U_q(\bo)$-module obtained from
$V$ by twisting by $\tau_a$. Then $V_a$ is in the category
$\mathcal{O}^*$ and 
\begin{equation}\label{shifttr}t_{V_a}(z,u) = t_V(za,u).\end{equation} 
Hence for any
$a\in\CC^*$ we have
$$t_V(za,u)t_{V'}(z,u) = t_{V'}(z,u) t_V(za,u).$$
This implies the result.
\end{proof}
For $V$ in category $\mathcal{O}^*$ with weights in $\overline{P}$, we
will denote
$$t_V(z,u) = \sum_{m\geq 0} t_V[m](u) z^m.$$

\subsection{Deformation of $U_q(\bo^-)^0$}

Let $t(v)$ be the $\CC[[v_i]]_{i\in I}$-subalgebra of
$U_q(\mathfrak{b}^-)[[v_i]]_{i\in I}$ generated by the
$t_{R_{i,a}^+}[m](u)$ ($i\in I$, $a\in\CC^*$, $m\geq 0$) and the
$t_{L(\overline{\omega})}[m](u)$ ($\omega\in Q^+$).

\begin{prop}\label{defo} $t(v)$ is a commutative subalgebra of
  $U_q(\mathfrak{b}^-)[[v_i]]_{i\in I}$ which is a deformation of
  $U_q(\bo^-)^0$.

More precisely, let $i\in I$. Then the limit at $v_j \rightarrow 0$
($j\in I$) of $t_{R_{i,a}^+}(z,u)$ is $T_i(za)$ where $$T_i(z) =
\text{exp}\left( \sum_{m > 0} z^m
  \frac{\tilde{h}_{i,-m}}{[d_i]_q[m]_{q_i}}\right)$$ and
$$\tilde{h}_{i,-m} = \sum_{j\in I}[d_j]_q \tilde{C}_{j,i}(q^m) h_{j,-m}.$$ 
\end{prop}

The commutativity is a direct consequence of Theorem \ref{com}. The
rest of the Proposition will be proved in Section \ref{proofdefo}.

\begin{example} In the case $\gb = sl_2$, we have $T_1(z) = \text{exp}\left(
    \sum_{m > 0} z^m \frac{h_{1,-m}}{[m]_q(q^m +
      q^{-m})}\right)$.\end{example}
The next lemma follows from \cite[Lemma 3.1]{Fre2}.
\begin{lem}\label{commti}  $T_i(z)$ commutes with the
$x_{j,r}^\pm$ for $j\neq i$ and $r\in\ZZ$.
\end{lem}

Let $W$ be a tensor product of simple objects in $\mathcal{C}$. It has
a highest weight vector $w$ of highest $\ell$-weight $m = \prod_{j\in I,
    b\in\CC^*}Y_{j,b}^{u_{j,b}} $. Then $w$
is an eigenvector of $T_i(z)$. Let us denote by $f_i(z)$ the corresponding
eigenvalue. A straightforward computation gives 
\begin{equation}\label{fi}f_i(z) = \prod_{j\in I,b\in\CC^*} 
\exp \left(u_{j,b}(m)\sum_{r > 0}(zb^{-1})^r
    \frac{\tilde{C}_{i,j}(q^r)}{r}\right).\end{equation}
By \cite[Theorem 4.1]{Fre2}, the other $\ell$-weights of $W$
are of the form
$$M = m A_{i_1,a_1}^{-1}A_{i_2,a_2}^{-1}\cdots A_{i_N,a_N}^{-1}$$
where $i_1,\cdots, i_N\in I$ and $a_1,\cdots, a_N\in\CC^*$. 

\begin{prop}\label{firstpol} The eigenvalue of $T_i(z)$ on the
  $\ell$-weight space $W_M$ is equal to
$$f_i(z)\times \prod_{1\leq k\leq N, i_k = i} (1 - za_k^{-1}).$$
\end{prop}

\begin{proof} Let us replace in $T_i(z)$ each $\tilde{h}_{j,-m}$ 
by the associated eigenvalue encoded by the monomial $A_{k,b}^{-1}$
($k\in I$, $b\in\CC^*$). We get
$$\text{exp}\left(-\sum_{j\in I, m > 0} z^m 
 \frac{\tilde{B}_{i,j}(q^m)[d_j]_q}{[m]_q} C_{j,k}(q^m) \frac{(q_j^m -
   q_j^{-m})}{m(q_j - q_j^{-1})}b^{-m} \right)$$
$$ = \text{exp}\left(-\sum_{j\in I, m > 0} (zb^{-1})^m 
  \frac{\tilde{B}_{i,j}(q^m)B_{j,k}(q^m)}{m} \right) =
\text{exp}\left(- \sum_{m > 0} \frac{(zb^{-1})^m
    \delta_{i,k}}{m}\right) = (1 - zb^{-1})^{\delta_{i,k}}.$$
Hence the result.\end{proof}

\subsection{Baxter polynomiality}
Let $W$ be a tensor product of simple objects in the category
$\mathcal{C}$ of finite-dimensional representations of
$U_q(\Glie)$. It has a highest weight vector $w$ of weight
$\omega$. For $i\in I$, let $f_i(z)\in\CC[[z]]$ be the eigenvalue of
$T_i(z)$ on $w$ as given in (\ref{fi}).  Let $\lambda$ be a weight of
$W$ and $ht_i(\omega -
\lambda)$ be the multiplicity of $\alpha_i$ in $\omega - \lambda$;
that is
$$
\omega - \lambda = \sum_{i \in I} ht_i(\omega -
\lambda) \cdot \alpha_i.
$$
One of the main results of this paper is the following Baxter
polynomiality.

\begin{thm}\label{catOpol} Let $i\in I, a\in\CC^*$ and $V =
  R_{i,a}^+$. Then the operator
$$(f_i(az))^{-1}t_V(z,u)\in
((\text{End}((W)_\lambda))[[v_j]]_{j\in I})[z]$$
is a polynomial in $z$ of degree $ht_i(\omega - \lambda)$. 
\end{thm}

This Theorem \ref{catOpol} will be proved in Section \ref{gradation}
and Section \ref{endproof}. By Theorem \ref{com} and Proposition
\ref{defo}, it implies immediately the following:

\begin{cor}\label{catOcase} Let $i\in I$, $a\in\CC^*$ and $V =
  R_{i,a}^+$. Then the eigenvalues of
$t_V(z,u)$ on $(W)_\lambda$ are of the form
$$f_i(za) Q_i(za,u),$$
where $Q_i(z,u)$ is a polynomial in $z$ of degree $ht_i(\omega -
\lambda)$. \end{cor}

We now give some applications of these results.

\subsection{Proof of the conjecture of Reshetikhin and the first author}

Our main application is the proof of a version of the conjecture of
Reshetikhin and the first author from \cite{Fre} about the spectra of
transfer-matrices. It is a consequence of Theorem \ref{catOpol} and
Corollary \ref{quotient}. Let us use the notation of the previous
section.

\begin{thm}\label{fdcase} Let $W$ be as above and $V$ a
  finite-dimensional representation of $U_q(\Glie)$. Then every
  eigenvalue of $t_V(z,u)$ on $(W)_\lambda$ may be expressed as
  $\chi_q(V)$ in which we replace each variable $Y_{i,a}$ by
\begin{equation}    \label{QQ}
q_i^{ht_i(\omega - \lambda)} a_iu_i \frac{f_{i}(azq_i^{-1})
Q_i(azq_i^{-1},u)}{f_{i}(azq_i)Q_i(azq_i,u)},
\end{equation}
where the series $f_{i}(z)\in\CC[[z]]$ given by formula (\ref{fi}) and
the numbers
\begin{equation}    \label{ai}
a_i = \prod_{j\in I}q_j^{-(C^{-1})_{j,i} \omega(\alpha_i^\vee)},
\qquad i \in I,
\end{equation}
depend only on $W$ (and not on the eigenvalue), whereas $Q_i(z,u)$ is
a polynomial in $z$ of degree $ht_i(\omega - \lambda)$ that depends on
the eigenvalue.
\end{thm}

The polynomials $Q_i(z,u)$ are the analogues of the Baxter polynomials
(see the Introduction for more details).

For $\omega = \lambda$, Theorem \ref{fdcase} follows from the
  definition of $\chi_q(V)$, as explained in \cite[Section
6.1]{Fre}. In this case, $Q_i(z,u)=1$ for all $i\in I$.

\begin{proof} By Theorem \ref{com}, it suffices to replace in Baxter's
  formula of Corollary \ref{quotient} the classes of representations
  $R_{i,a}^+$, $L(\overline{\omega_i})$ by the eigenvalues of their
  respective twisted transfer-matrices. For $R_{i,a}^+$, we use the
  formula in Corollary \ref{catOcase}. For $L(\overline{\omega_i})$, 
  by Example \ref{extr}, we
  have to compute the eigenvalue of $u_i\tilde{k}_i^{-1}$. Let
  $a_i$ be the eigenvalue of $\tilde{k}_i^{-1}$ on $w$. We get the
  formula $a_i u_i q_i^{ht_i(\omega -\lambda)}$ because for $j\in I$,
  $r\in \ZZ$ we have $\tilde{k}_i x_{j,r}^- = q_i^{-\delta_{i,j}}
  x_{j,r}^- \tilde{k}_i$.
\end{proof}

\begin{rem}\label{remfdcase} (i) Suppose that $V$ is irreducible and
let $m_V$ be its highest weight dominant monomial.
By \cite[Theorem 4.1]{Fre2}, any monomial in $\chi_q(V)$ has the form
$$M_V = m_V \prod_{k=1}^N A_{i_k,a_k}^{-1}.$$
Let us write out $M_V$ as the product of the $Y_{i,a}$, and then
replace each $Y_{i,a}$ by the corresponding factor $a_i u_i
f_i(zaq_i^{-1})/f_i(zaq_i)$ appearing in Theorem \ref{fdcase}. We
obtain a scalar function, which we denote by $F_{M_V}$. A
straightforward computation using formula (\ref{fi}) yields that the
ratio between $F_{M_V}$ and $F_{m_V}$ is given by the following
rational function in $z$:
$$
\frac{F_{M_V}}{F_{m_V}} =
%\prod_{k=1}^N\prod_{b\in\CC^*}
%\left(\frac{1-zb^{-1}a_kq_{i_k}^{-1}}{1-zb^{-1}a_kq_{i_k}}
%\right)^{u_{i_k,b}(m)}
%=
\prod_{k=1}^N
v_{i_k}^{-1} q_{i_k}^{-\on{deg}(P_{i_k})}
\frac{P_{i_k}(z^{-1}a_k^{-1}q_{i_k})}{P_{i_k}(z^{-1}a_k^{-1}q_{i_k}^{-1})},
$$
where the $P_j, i \in I$, are the Drinfeld polynomials of the highest
weight vector $w$ of $W$.
%where $m$ is the highest monomial of $W$ as in formula (\ref{fi})
%above and the $P_j$ are the corresponding Drinfeld polynomials
%(although this statement is technically independent from Proposition
%\ref{firstpol}, it is obtained by a similar computation).

Therefore, it follows from Theorem \ref{fdcase} that all eigenvalues
of $t_V(z,u)$ on $W$ are rational functions in $z$ up to one and the
same overall scalar factor, namely $F_{m_V}=$
%$$F_{m_V} = \prod_{i,j \in I, a\in\CC^*} (v_jq_j^{-\omega(\alpha_i^\vee)})^{(C^{-1})_{j,i} u_{i,a}(m_V)} \text{exp}\left( \sum_{r > 0, b\in\CC^*}u_{i,a}(m_V)u_{j,b}(m)(zab^{-1})^r (q_i^{-r} - q_i^r)\tilde{C}_{i,j}(q^r)  \right)$$
$$\prod_{i,j \in I, a\in\CC^*} \left((v_jq_j^{-\omega(\alpha_i^\vee)})^{(C^{-1})_{j,i} } \exp\left( \sum_{r > 0, b\in\CC^*}\frac{u_{j,b}(m)(zab^{-1})^r (q_i^{-r} - q_i^r)\tilde{C}_{i,j}(q^r)}{r}  \right)\right)^{u_{i,a}(m_V)}
$$
where $u_{i,a}(m_V)$ (resp. $u_{i,a}(m)$) 
is the power of $Y_{i,a}$ in $m_V$ (resp. in the highest monomial $m$ of $W$).

%(which is the product of the
%factors $f_l(zbq_l^{-1})/f_l(zbq_l)$, with $f_i(z)$ given by formula
%(\ref{fi}), one for each factor $Y_{l,b_l}$ appearing in $m$).

This is in agreement with the fact that up to a scalar function the
operator $\pi_W(t_V(z))$ is rational in $z$ (see \cite{ifre} and
\cite[Proposition 9.5.3]{efk}). Our result gives a description of
the eigenvalues of the $u$-deformation of this operator.

%In the particular case $M = m A_{i,1}^{-1}$, multiplying $F_{m_V}/F_M$
%by the factor $\prod_{j\in I} (u_j a_j)^{C_{j,i}} = v_{i}
%q_{i}^{-\on{deg}(P_{i})}$ of Theorem \ref{fdcase}
%and specializing $v_{i} = q_{i}^2$ as in \cite{Fre}, we get $q_{i}^{2
%  + \on{deg}(P_{i})}
%\frac{P_{i}(z^{-1}q_{i}^{-1})}{P_{i}(z^{-1}q_{i} )}$ where
%$P_{i}$ is the Drinfeld polynomial corresponding the monomial $m$ of
%$W$.  We recover the formula for $A_{i,1}$ in

This is also in agreement with the calculations in \cite[Section
6]{Fre}.

(ii) It follows from the definition that each $z^m$-coefficient of
$Q_i(z,u)$ is a root of a polynomial whose
coefficients are in the ring of formal Taylor power series in the
$v_j$, $j\in I$ (hence it belongs to the algebraic closure of the field
of fractions of this ring). We expect these series to be expansions of
rational functions near
the point $v_j = 0$, $j\in I$. It would be interesting to prove that
this is indeed the case and to describe the poles of these rational
functions.
\end{rem}

\begin{example} (i) For $\gb = sl_2(\CC)$ and $V = L(Y_{1,q^{-1}})$,
  we get as in (\ref{relB})
$$D(z)\frac{
Q_1(zq^{-2},u)}{Q_1(z,u)} + A(z)
\frac{Q_1(zq^2,u)}{Q_1(z,u)},$$
$$\text{with }A(z) = (D(zq^2))^{-1} = a_1^{-1}u_1^{-1}q^{-ht_1(\omega - \lambda)}
\frac{f_1(zq^2)}{f_1(z)}.$$
Dividing by $D(z) = F_{Y_{1,q^{-1}}}(z)$ we get a rational fraction in
$z$ (see Remark \ref{remfdcase} (i)):
$$\frac{
Q_1(zq^{-2},u)}{Q_1(z,u)} +
v_1^{-1}q^{-\on{deg}(P)-2ht_1(\omega - \lambda)}\frac{P(z^{-1}q)}{P(z^{-1}q^{-1})}
\frac{Q_1(zq^2,u)}{Q_1(z,u)},$$
where $P$ is the Drinfeld polynomial corresponding to the highest
monomial $m$ of $W$.

(ii) In general there are more than $2$ terms. For $\gb = sl_3(\CC)$
and $V = L(Y_{1,q^{-1}})$, we get
$$D_1(z)\frac{Q_1(zq^{-2},u)}{Q_1(z,u)} + (D_1(zq^2))^{-1}D_2(z)
\frac{Q_1(zq^2,u)Q_2(zq^{-1},u)}{Q_1(z,u)Q_2(zq,u)} +
(D_2(zq^2))^{-1} \frac{Q_2(zq^3,u)}{Q_2(zq,u)},$$
$$\text{with }D_1(z) = a_1u_1q^{ht_1(\omega -
  \lambda)}\frac{f_1(zq^{-2})}{f_1(z)}, \quad \text{and} \quad D_2(z) =
a_2u_2q^{ht_2(\omega - \lambda)}\frac{f_2(zq^{-1})}{f_2(zq)}.$$
Dividing by $D_1(z) = F_{Y_{1,q^{-1}}}(z)$ we get a rational fraction
in $z$ (see Remark \ref{remfdcase} (i)):
$$\frac{Q_1(zq^{-2},u)}{Q_1(z,u)} +
v_1^{-1}q^{-\on{deg}(P_1)+(ht_2 - 2ht_1)(\omega - \lambda)}\frac{P_1(z^{-1}q)}{P_1(z^{-1}q^{-1})}
\frac{Q_1(zq^2,u)Q_2(zq^{-1},u)}{Q_1(z,u)Q_2(zq,u)}$$
 $$+
v_1^{-1}v_2^{-1}q^{-\on{deg}(P_1P_2)-(ht_1 + ht_2)(\omega - \lambda)}
\frac{P_1(z^{-1}q)P_2(z^{-1})}{P_1(z^{-1}q^{-1})P_2(z^{-1}q^{-2})}
\frac{Q_2(zq^3,u)}{Q_2(zq,u)},$$
where $P_1,P_2$ are the Drinfeld polynomials corresponding to the
highest monomial $m$ of $W$.
\end{example}

\subsection{Bethe Ansatz equations}\label{bae}

We now derive the generalized Bethe Ansatz equations from
Theorem \ref{fdcase} following \cite[Section 6.3]{Fre}.

According to Theorem \ref{fdcase}, each eigenvalue of $t_V(z,u)$ on a
finite-dimensional representation $W$ is a sum of terms, each having
the product of the functions of the form $f_{i}(azq_i)Q_i(azq_i,u)$ in
the denominator.

Suppose that $W = \otimes_{j=1}^N L(\Psib_j)$, where
\begin{equation}    \label{Psij}
\Psib_j = (\Psi_{j,i}(z))_{i\in I},
\quad
\Psi_{j,i}(z) = 
q_i^{\on{deg}(P_i)}\frac{P_{j,i}(zq_i^{-1})}{P_{j,i}(zq_i)},
\end{equation}
where $P_{j,i}(z)$ are polynomials (see Section \ref{fdrep}).  The
zeros of $f_{i}(azq_i)$ differ from the roots of $P_{j,i}(z)$ by
powers of $q$, and therefore the corresponding poles in the
eigenvalues of $t_V(z,u)$ on $W$ are to be expected. But the roots
of $Q_i(azq_i,u)$ give rise to extraneous poles, which
we do not expect to have in the eigenvalues of $t_V(z,u)$. Therefore
they should cancel each other, and this must happen uniformly for all
$V$.

We expect that, at least for generic values of $q$, the only possible
way for this to happen is for the poles in the monomials of the form
$M$ and $M A_{i,aq_i}^{-1}$ in the $q$-character of $V$ to cancel
out.

We also expect that each root of $Q_i(z,u)$ has multiplicity one. Then
each cancellation of this type gives rise to an equation, which says
that the sum of the residues of the terms in the eigenvalues
corresponding to $M$ and $M A_{i,aq_i}^{-1}$ at the poles coming from
the roots of $Q_i(zaq_i,u)$ is equal to $0$.

To write down these equations explicitly, let us set
$$
Q_i(z,u) = \prod_{k=1}^{m_i} (w^{(i)}_k-z), \qquad m_i=ht_i(\omega -
\lambda).
$$
Recall that $Q_i(z,u)$ is a polynomial in $z$ whose coefficients are
in the algebraic closure of the field of fractions of the ring of
formal Taylor power series in the $v_i, i \in I$. Hence each root
$w^{(i)}_k$ belongs to the same field. Further, since the $Q_i(z,u)$
enter the eigenvalues through the ratios \eqref{QQ}, we do not lose
any generality by normalizing $Q_i(z,u)$ this way.

An explicit calculation along the lines of those in \cite[Section
6]{Fre} gives us the following system of equations on the $w^{(i)}_k$:
\begin{equation}    \label{bae gen}
v_i \prod_{j=1}^N q_i^{\deg P_{j,i}}
  \frac{P_{j,i}(q_i^{-1}/w^{(i)}_k)}{P_{j,i}(q_i/w^{(i)}_k)} =
  \prod_{s \neq k} q_i^2
  \frac{w^{(i)}_k-w^{(i)}_sq_i^{-2}}{w^{(i)}_k-w^{(i)}_s q_i^{2}}
  \prod_{l \neq i} \prod_{s=1}^{m_l} q^{C_{li}}
  \frac{w^{(i)}_k-w^{(l)}_sq^{-C_{li}}}{w^{(i)}_k-w^{(l)}_s
  q^{C_{li}}}.
\end{equation}

These are the generalized Bethe Ansatz equations corresponding to a
given collection of polynomials $P_{j,i}, j=1,\ldots,N; i \in I$.

These equations come from ``local'' pole cancellations, in the sense
that they occur between monomials of the form $M$ and $M
A_{i,aq_i}^{-1}$. Therefore the equations are the same for
any choice of the representation $V$.

If ${\mathfrak g} = \wh{sl}_2$ and $W = \otimes_{j=1}^N
W^{(1)}_{R_j,b_jq^{-R_j+1}}$ (we use the
notation of Example \ref{exkr}), then these equations become
\begin{equation}    \label{bae sl2}
v \prod_{j=1}^N q^{R_j} \frac{w_k-b_j q^{-R_j}}{w_k-b_j q^{R_j}}
= \prod_{s\neq k} q^2 \frac{w_k-w_sq^{-2}}{w_k-w_s q^{2}}, \qquad
k=1,\ldots,ht(\omega - \lambda).
\end{equation}

Formulas (\ref{bae gen}) and (\ref{bae sl2}) specialize to formulas
(6.6) and (6.5) of \cite{Fre}, respectively, if we set $v_i=q_i^2$
(the factor $-q^{-N}$ on the RHS of formula (6.5) in \cite{Fre} should
be replaced by $q^{2m-4}$, and similarly for formula (6.6) in
\cite{Fre}).

\begin{example}\label{expequaun} For $N = R_1 = 1 = k$ and $b_1 = q^{-1}$, 
we recover the well-known relations
$$v q \frac{w_1-q^{-2}}{w_1 - 1} = 1\text{ and }w_1 = \frac{1 -
  vq^{-1}}{1 - vq}.$$ 
\end{example}

%In the case of $\wh{sl}_2$, we can restrict ourselves, without loss of
%generality, to $V$ being the two-dimensional representation
%$W^{(1)}_a$. Then $\chi_q(V)$ has two terms: $Y_{1,a}$ and
%$Y_{1,aq^2}^{-1}$. Therefore, the only way the extraneous poles can
%cancel each other is for the residues at those poles in the two terms
%to be opposite to each other. This implies the Bethe Ansatz equations
%(\ref{bae sl2}). Therefore we obtain the following result.

%\begin{thm}
%  Let ${\mathfrak g} = \wh{sl}_2$ and $W= \otimes_{j=1}^N
%  W^{(1)}_{R_j,b_jq^{-R_j+1}}$. Then for generic $q$ and generic $b_j$,
%  each common eigenvalue of the transfer-matrices $t_V(z,u)$ on
%  $(W)_\lambda$ gives rise to a solution of the Bethe Ansatz equations
%  (\ref{bae sl2}).
%\end{thm}

%We conjecture that conversely, any solution of the Bethe Ansatz
%equations (\ref{bae sl2}) gives rise to a common eigenvalue of the
%transfer-matrices on $(W)_\lambda$, so that there is a one-to-one
%correspondence between the solutions and the eigenvalues. Moreover, we
%expect this to hold for all quantum affine algebras.

We expect that for generic $q$ and generic polynomials $P_{j,i}$, it
is possible prove along the lines of the above argument that any
eigenvalue of the transfer-matrices $t_V(z,u)$ on $(W)_\lambda$ gives
rise to a solution of the Bethe Ansatz equations (\ref{bae
  sl2}). Furthermore, we expect that the converse is true as
well. Thus, we arrive at the following conjecture, which is a version
of the ``completeness of Bethe Ansatz'' (at the level of eigenvalues).

\begin{conj}
  For generic $q$ and generic polynomials $P_{j,i}$, there is a
  one-to-one correspondence between the eigenvalues of the
  transfer-matrices $t_V(z,u)$ on $(W)_\lambda$, where $W =
  \otimes_{j=1}^N L(\Psib_j)$, with $\Psib_j$ given by formula
  (\ref{Psij}) and the solutions of the Bethe Ansatz equations
  (\ref{bae gen}) with $w^{(i)}_k, k=1,\ldots,ht_i(\omega - \lambda);
  i \in I$.
\end{conj}

\subsection{Example}\label{sl2ex} We suppose that $\gb = sl_2$, $V =
R_{1,1}^+$, and $W = W_{N,q^{1-2N}}^{(1)}$ is a KR-module (we use the
notation of Example \ref{exkr}). There is an explicit formula for the
$R$-matrix (see Section \ref{fact} below). We have
$(\pi_W(x_0^-))^{N+1} = (\pi_W(kx_{-1}^+))^{N+1} = 0$.  So the image
$\mathcal{L}_V(z)$ of $L_V(z)$ in
$(\text{End}(V)\otimes\text{End}(W))[[z]]$ is a product
$$\mathcal{L}_V(z) = \mathcal{L}_V^+(z)(\text{Id}_V\otimes
\pi_W(T(z)))\mathcal{L}_V^-(z)\mathcal{L}^\infty$$
$$\mathcal{L}_V^+(z) = \sum_{0\leq r\leq N} \frac{((q^{-1} -
  q)\pi_V(x_{1,0}^+)\otimes \pi_W(x_{1,0}^-))^r}{q^{\frac{r(r-1)}{2}}[r]_q!},$$
$$\mathcal{L}_V^-(z) = \sum_{0\leq r\leq N} \frac{((q - q^{-1})z
  \pi_V(k_1^{-1}x_{1,1}^-)\otimes
  \pi_W(x_{1,-1}^+)k_1)^r}{q^{\frac{r(r-1)}{2}}[r]_q!}$$
$\mathcal{L}_V^\infty = (\pi_V\otimes \pi_W) (\mathcal{R}^\infty)$ does
not depend on $z$.

\noindent By taking the twisted trace, the image of the twisted
transfer-matrix in $(\text{End}(W))[[u, z]]$ is
$$\sum_{0\leq r\leq N} \frac{(-(q -
  q^{-1})^2z)^r}{q^{r(r-1)}([r]_q!)^2}  (Tr_{V,u}\otimes \pi_W)(
(x_{1,0}^+\otimes x_{1,0}^-)^r(1\otimes
T_1(z))(k_1^{-1}x_{1,1}^-\otimes x_{1,-1}^+ k_1)^r
\mathcal{L}_V^\infty)$$
$$=\sum_{0\leq r\leq N} \frac{((q - q^{-1})z)^r}{[r]_q!}
\pi_W((x_{1,0}^-)^r T_1(z) (x_{1,-1}^+ k_1)^r) \sum_{m\geq r}
u^{2m}\begin{bmatrix}m\\r\end{bmatrix}_q
q^{\frac{r(3-r)}{2}-rm}
\pi_W(k_1^{-m}).$$
Let $0\leq j\leq N$. Note that $h_{1,-m}.w_0 =
\frac{q^{Nm}[Nm]_q}{m}w_0$ for $m > 0$.
Then $(x_{1,-1}^+ k_1)^r k_1^m.w_j = q^{r(N-2)-m(N-2j)} w_{j-r}$
is an eigenvector of $T_1(z)$
with eigenvalue $f_1(z) P_{j-r}(z)$ where 
$$f_1(z) = \text{exp}\left(\sum_{m > 0}z^m
  \frac{q^{Nm}[Nm]_q}{m[m]_q(q^m + q^{-m})}\right)$$
does not depend of $r,j$ and 
$$P_{j-r}(z) = (1 - z)(1 - zq^2)\cdots (1 - zq^{2(j-r-1)})$$ 
is a polynomial in $z$ of degree $j - r$ (this can be computed by hand
or with Proposition \ref{firstpol}). Since $(x_{1,0}^-)^r w_{j-r} =
\frac{[j]_q![N-j+r]_q!}{[j-r]_q![N-j]_q!} w_j$, this implies that
$w_j$ is an eigenvector of $(f_1(z))^{-1}t_V(z,u)$ with eigenvalue
$Q_1(z,u)$ equal to
$$\sum_{0\leq r\leq N}
\frac{(q - q^{-1})^r[N-j+r]_q!\begin{bmatrix}j\\r\end{bmatrix}_qz^r(1
  - z)\cdots (1 - zq^{2(j-r-1)})}{q^{r(\frac{r+1}{2} - N )}[N-j]_q!}
\sum_{m\geq r}
u^{2m}\begin{bmatrix}m\\r\end{bmatrix}_q q^{-m(r+N-2j)}$$ which is a
polynomial in $z$. It is
clear that the degree is at most $j$. By taking the limit $u = 0$,
only the term with $m = r = 0$ contributes and we see that the degree
is exactly $j$.

Note that in addition each $z^m$-coefficient of
  $Q_1(z,u)$ is rational in $v = u^2$.

\begin{example}\label{exppol1} For $N = j = 1$, we get the well-known Baxter polynomial
$$Q_1(z,u) = \frac{(1-z)(1-u^2q^{-1}) + zu^2(q - q^{-1})}{(1 - u^2q)(1
  - u^2q^{-1})}.$$
It has degree $1$ and its root is $w_1 = \frac{1 - u^2 q^{-1}}{1 -
  u^2q}$ which specializes at $(1 + q + q^2)^{-1}$
for $u^2 = v = q^2$. This is the same as in Example \ref{expequaun}
above. It means $w_1$ is not a pole of
$$q \frac{Q_1(zq^{-2},u)}{Q_1(z,u)} + q^{-2} u^{-2} \frac{1 -
  z^{-1}}{1 - z^{-1}q^{-2}}\frac{Q_1(zq^2,u)}{Q_1(z,u)}.$$
\end{example}

\subsection{Polynomiality of Drinfeld's Cartan elements}\label{secjordan}
Let us now give the second application of our main results.

By Proposition \ref{defo}, Corollary \ref{catOcase} implies:
\begin{thm}\label{jordan} Let $i\in I$ and $f_{i}(z)\in\CC[[z]]$ the
  eigenvalue of $T_i(z)$ on a highest weight vector of $W$.
Then on $(W)_\lambda$ the operator
$$(f_{i}(z))^{-1}T_i(z)\in  (\text{End}((W)_\lambda))[z]$$ 
is a polynomial in $z$ of degree $ht_i(\omega - \lambda)$. 
\end{thm}
Note that this is compatible with Proposition \ref{firstpol} which
gives the polynomiality, up to $f_i(z)$, for the eigenvalues (the
result here is much stronger as we get the polynomiality also for the
off-diagonal elements).

As an illustration, let us check this result by hand in the most
elementary not trivial case. Suppose that $\gb = sl_2$. Let $$V =
L(Y_{1,1})\otimes L(Y_{1,1}).$$ An highest weight vector of $V$ is an
eigenvector of $T_1(z)$. Let $g(z)$ be the eigenvalue.  We study the
action on the $2$-dimensional space $V_0$. $g(z)^{-1}T_1(z)$ has a
unique eigenvalue on $V_0$ which is $(1-zq^{-1})$.  Let us study the
representation as in \cite[Examples 2.2, 3.3]{h3}.  From the relation
$[x_{1,-1}^+,x_{1,0}^-] = - \phi_{1,-1}^-/(q - q^{-1})$, we get
$h_{1,-1} = q^2f_0f_1 - f_1f_0$.  Let $w^+$ be a highest weight vector
of $L(Y_{1,1})$. Then $(w^+, w^- = f_1w^+)$ is a basis of
$L(Y_{1,1})$.  We use analogous notation $(v^+,v^- = f_1v^+)$ for a
second copy of $L(Y_{1,1})$.  Then $(w^-\otimes v^+, q^{-2} w^-\otimes
v^+ + w^+\otimes w^-)$ is a basis of $V_0$. In this basis, for $m > 0$
the action of $h_{1,-m}$ is the matrix
$\begin{pmatrix}\frac{[m]_q}{m}(1-q^{-2m})&a_m\\0&\frac{[m]_q}{m}(1-q^{-2m})\end{pmatrix}$
where $a_m\in\CC^*$ and $a_1 = q - q^{-4}$. For $r \geq 0$, the vector
$x_{1,-r}^-.(w^+\otimes v^+)$ in this basis has the form
$\begin{pmatrix}\lambda_r \\ \mu_r\end{pmatrix}$ where
$\lambda_r,\mu_r\in\CC$ and $\lambda_0 = 0$, $\mu_0 = 1$. We have the
relation $[h_{1,-m},x_{1,-r}^-] = - \frac{[2m]_q}{m}x_{1,-m - r}^-$
for $m > 0$, $r\geq 0$. Since $h_{1,-m}.(w^+\otimes v^+) =
\frac{2[m]_q}{m}(w^+\otimes v^+)$, it implies
$$\begin{cases}-\frac{[m]_q}{m}(1+q^{-2m})   \lambda_r +  a_m \mu_r  =
  - \frac{[2m]_q}{m} \lambda_{m+r},
  \\ [m]_q(1+q^{-2m}) \mu_r = [2m]_q \mu_{m+r}.\end{cases}$$ The
second equation with $m = 1$ implies $\mu_r = q^{-r}$ for $r\geq
0$. Then the first equation with $m = 1$ reads $- q^r \lambda_r +
\frac{q^2 - q^{-3}}{q + q^{-1}} = - q^{1+r} \lambda_{1+r}$ which
implies $\lambda_r = q^{-r} r \frac{q^{-3} - q^2}{q + q^{-1}}$ for
$r\geq 0$. Now the first equation with $r = 0$ gives $a_m = - [2m]_q
q^{-m} \frac{q^{-3} - q^2}{q + q^{-1}}$. Hence the image of
$g(z)^{-1}T_1(z)$ in the basis is
$$(1 - zq^{-1}) \text{exp}\left(\sum_{m > 0}z^m
  \frac{\begin{pmatrix}0&- [2m]_q q^{-m}  \frac{q^{-3} - q^2}{q +
        q^{-1}}\\0&0\end{pmatrix}}{[m]_q(q^m + q^{-m})}\right)$$
$$= (1 - zq^{-1}) \begin{pmatrix}1&- \sum_{m > 0}z^m \frac{[2m]_q
    q^m}{[m]_q(q^m + q^{-m})}  \frac{q^{-3} - q^2}{q +
    q^{-1}}\\0&1\end{pmatrix}
= \begin{pmatrix}1 - zq^{-1}&z \frac{q - q^{-4}}{q + q^{-1}}\\0&1 -
  zq^{-1}\end{pmatrix}.$$ It is a polynomial of degree $1$ in $z$.

\subsection{Plan of the proof of Theorem \ref{catOpol}} We first
establish (Theorem \ref{defigrad}) a grading of positive
prefundamental representations $R_{i,a}^+$ which has good
compatibility properties with the action of $\ZZ$-graded elements in
$U_q(\bo)$. Section \ref{gradation} is entirely devoted to this
grading and the proof of its existence. We believe that the study of
this grading will be interesting independently of the applications it
finds in this paper.

Then in Section \ref{fact} we recall the factorization of the
universal $R$-matrix. We give the proof of Proposition \ref{defo} in
Section \ref{proofdefo}. Then we explain why it suffices to consider
the case that $W$ is a tensor product of fundamental representations
with $\ell$-weight spaces of dimension $1$ (or of dimension at most
$2$ for $\gb$ of type $E_8$). This is crucial to control the action of
the Cartan factor of the universal $R$-matrix. By using the grading of
Theorem \ref{defigrad}, we can also control the action of the positive
and negative parts, as we explain in Section \ref{endproof}.

\section{A grading of positive prefundamental
  representations}\label{gradation}

To prove Theorem \ref{catOpol}, we first establish the existence of a
certain grading on positive prefundamental representations with nice
properties with respect to the action of Drinfeld generators (Theorem
\ref{defigrad}). We believe this grading is also of independent
interest.

\subsection{The grading}

Let us fix $i\in I, a\in\CC^*$. The main result of this section is the
following.

\begin{thm}\label{defigrad} There exists a grading by
  finite-dimensional vector spaces
$$R_{i,a}^+ = \bigoplus_{m\in\ZZ} (R_{i,a}^+)_m$$
such that 

(1) for $m\geq 0$ and $x\in U_q(\bo)^-$ of degree $r > 0$, we have
$$x((R_{i,a}^+)_{m})\subset (R_{i,a}^+)_{m-r}.$$

(2) for $j\in I$, $r, m\geq 0$ we have
$$\phi_{j,r}^+ (R_{i,a}^+)_{m}\subset (R_{i,a}^+)_{m - r}\text{ if $j\neq i$},$$
$$(\phi_{i,r}^+ + a\phi_{i,r-1}^+ + \cdots + a^r
\phi_{i,0}^+)(R_{i,a}^+)_{m}\subset (R_{i,a}^+)_{m - r},$$

(3) For $j\in I$ and $r\geq 0$ we have 
$$x_{j,r}^+((R_{i,a}^+)_{m})\subset  (R_{i,a}^+)_{m -r} + (R_{i,a}^+)_{m -r + \delta_{i,j}}.$$
\end{thm}

This result will be proved in this section.

\begin{rem}\label{nilph} The condition (2) for $r = 0$ implies that
  the grading is compatible with weight decomposition:
$$(R_{i,a}^+)_\lambda = \bigoplus_{m\in\ZZ}(R_{i,a}^+)_\lambda\cap
(R_{i,a}^+)_m\text{ for $\lambda\in \tb^*$.}$$
It also implies that for $j\in I$, $r > 0$ we have
$$(h_{j,r} - \lambda_{j,r})(R_{i,a}^+)_{m}  \subset
(R_{i,a}^+)_{m-r}\text{ where }\lambda_{j,r} = \frac{\delta_{i,j}
  a^r}{r(q_i - q_i^{-1})}.$$
Up to a shift, we can assume that an $\ell$-lowest weight vector of
$R_{i,a}^+$ has degree $0$ and that $(R_{i,a}^+)_m = 0$ for $m < 0$.
\end{rem}

\subsection{Root vectors}\label{roots} 

Let us remind results from \cite{bec, da} where the root vectors
$E_\alpha\in U_q(\bo)$, $F_\alpha\in U_{q}(\bo^-)$ are
constructed for
$$\alpha\in \Phi_+^{Re} = \Phi_0^+\sqcup\{\beta +m\delta|m > 0, \beta \in \Phi_0\}.$$
Here $\Phi_0$ (resp. $\Phi_0^+$) is the set of roots (resp. positive
roots) of $\dot{\mathfrak{g}}$ and $\delta$ is the standard imaginary
root of $\mathfrak{g}$.
For example, we have for $i\in I$, $m \geq 0$ and $r > 0$:
$$E_{m\delta+\alpha_i} = x_{i,m}^+ ,\quad E_{r\delta-\alpha_i} =
-k_i^{-1} x_{i,r}^- , \quad F_{m\delta+\alpha_i} = x_{i,-m}^-, \quad
F_{r\delta-\alpha_i} = -x_{i,-r}^+k_i.$$
We will consider the subalgebras 
$$U_q(\bo^-)^{+,0} = \tb \otimes U_q(\bo^-)^+ \subset U_q(\bo^-)^{\geq
  0} = U_q(\bo^-)^0 \otimes U_q(\bo^-)^+,$$
$$U_q(\bo)^{-,0} =  U_q(\bo)^- \otimes \tb \subset U_q(\bo)^{\leq 0} =
U_q(\bo)^-\otimes U_q(\bo)^0.$$
The $\ZZ$-grading of $U_q(\g)$ induces $\ZZ$-gradings on
$U_q(\bo^-)^{+,0}$ and $U_q(\bo)^{-,0}$.

The subalgebra of $U_q(\bo)$ (resp. $U_q(\bo^-)$) generated by the
$E_{-\alpha + r \delta}$ (resp. $F_{-\alpha + r\delta}$) for
$\alpha\in \Phi_0^+$, $r> 0$ and by $\tb$ is $U_q(\bo)^{-,0}$
(resp. $U_q(\bo^-)^{+,0}$). Note that we have
$$\on{deg}(E_{-\alpha + r \delta}) = - \on{deg}(F_{-\alpha + r \delta}) = r.$$

\subsection{Examples for the grading}\label{simpleun} The examples
in this section are given as an illustration and are not used for the
main results of the paper.

First, for $\gb = sl_2$, we can check directly that we can choose 
$$(R_{1,a}^+)_m = \CC . v_m^*,\quad m\geq 0.$$
Indeed,

(1) for $m\geq 0$, we have $x_{1,1}^-((R_{1,a}^+)_{m}) =
(R_{1,a}^+)_{m-1}$ and $x_{1,r}^-((R_{1,a}^+)_{m}) = \{0\}$ if $r >
1$.

(2) For $r\geq 0$, $(\phi_{1,r}^+ + a\phi_{1,r-1}^+ + \cdots + a^r
\phi_{1,0}^+) =a^{r-1}(\delta_{r\neq 0}\phi_{1,1}^+ + a \phi_{1,0}^+)
=  \delta_{r,0} k_1$ on $R_{1,a}^+$.

(3) We have $x_{1,0}^+((R_{1,a})_m) = (R_{1,a})_{m+1}$ and
$x_{1,r}^+((R_{1,a})_m) = \{0\}$ for $r > 0$.

\noindent This can be generalized to the case $N_i = 1$. In this
special case, there is a simple proof thanks to the following result.

\begin{thm}\label{asympt}\cite{HJ} Suppose $N_i = 1$. Let $j\in I$, $r
  > 0$, $\alpha \in \Phi_0^+$.

  (1) $\phi_{j,\delta_{i,j} + r}^+$ acts by $0$ and
  $k_i^{-1}\phi_{i,1}^+$ is a scalar operator on $R_{i,a}^+$.

(2) $x_{j,r}^+$ acts by $0$ on $R_{i,a}^+$. 

(3) If $\alpha(\alpha_i^\vee) = 0$, then $E_{-\alpha + r\delta}$ acts
by $0$ on $R_{i,a}^+$.

(4) If $\alpha(\alpha_i^\vee) = 1$, then $E_{-\alpha + (r + 1)\delta}$
acts by $0$ on $R_{i,a}^+$.
\end{thm}

\begin{rem} Precisely, the statement is proved in \cite[Section
  7.2]{HJ} for $L_{i,a}^+$ by giving an asymptotic construction of
  $L_{i,a}^+$. The same construction works for
  $\overline{L}_{i,a}^+$. By using $\hat{\omega}$, this gives also an
  asymptotic construction of $R_{i,a}^+$. Hence the result.\end{rem}

\begin{cor}\label{corasy} If $N_i = 1$, there is $p\in\ZZ$ such that for any $m\in\ZZ$,
$$(R_{i,a}^+)_m = \bigoplus_{\{\alpha\in Q^+|\alpha(\alpha_i^\vee) = m
  + p\}} (R_{i,a}^+)_{\alpha}.$$
\end{cor}

\begin{proof} Let $p$ be the degree of a lowest weight vector.  Let
  $w\in (R_{i,a}^+)_{\alpha}$ be a non zero weight vector.  Then by
  Section \ref{roots} there is a non zero lowest weight vector of the
  form $E_{-\alpha_{i_1}+r_1\delta}\cdots E_{-\alpha_{i_N} +
    r_N\delta}w$.  Hence $m = r_1 + \cdots + r_N + p$ and $\alpha =
  \alpha_{i_1} + \cdots + \alpha_{i_N}$. But by Theorem \ref{asympt},
  $r_1 = \cdots = r_N = 1$ and $\alpha_{i_1}(\alpha_i^\vee) = \cdots =
  \alpha_{i_N}(\alpha_i^\vee) = 1$. So $m = N + p =
  \alpha(\alpha_i^\vee) + p$.
\end{proof}

In general, the statements of Theorem \ref{asympt} and Corollary
\ref{corasy} do not hold (see the following example).  The reason is
that the asymptotic construction in \cite{HJ} do not work in these
cases.  That is why we give a completely different proof in the next
subsections.

\begin{example} Let us consider the $B_2$-case with $i = 1$ the node
  satisfying $N_1 = 2$.  Let $v$ be a lowest weight vector of $V =
  R_{1,1}^+$. By \cite[Theorem 6.4]{HJ}, we have
  $\text{dim}(V_{2\alpha_1 + \alpha_2}) = 3$.  If the statement of
  theorem \ref{asympt} held for $V$, we would have
$$V_{2\alpha_1 + \alpha_2} = \CC (x_{1,0}^+)^2x_{2,0}^+.v \oplus \CC
x_{1,0}^+x_{2,0}^+x_{1,0}^+.v \oplus \CC x_{2,0}^+ (x_{1,0}^+)^2.v.$$
But $(x_{1,0}^+)^2x_{2,0}^+.v = 0$, contradiction.  The statement of
Corollary \ref{corasy} does not hold neither. Consider a grading
such that $v$ has degree $0$.  The weight spaces of weight $0$,
$\alpha_1$, $2\alpha_1$, $\alpha_1 + \alpha_2$ are of dimension $1$,
generated respectively by $v$, $x_{1,0}^+.v$, $(x_{1,0}^+)^2.v$,
$(x_{2,0}^+x_{1,0}^+).v$, and have respective degree $0$, $1$, $2$,
$1$.  $V_{2\alpha_1 + \alpha_2}$ is generated by $v_1 =
(x_{1,0}^+x_{2,0}^+x_{1,0}^+).v$, $v_2 = x_{2,0}^+(x_{1,0}^+)^2.v$ and
$v_3 = x_{2,1}^+(x_{1,0}^+)^2.v = - x_{1,1}^+x_{2,0}^+x_{1,0}^+.v$. By
construction, $v_3$ has degree $1$ and $v_2$ has degree $2$.  Since
$x_{1,1}^-(v_1 + \CC v_2 + \CC v_3)\subset (V)_1$, there are $\lambda,
\mu\in\CC$ such that $v_1 + \lambda v_2 + \mu v_3$ has degree $2$.
Hence $V_{2\alpha_1 + \alpha_2}\cap (V)_1$ has dimension $1$ and
$V_{2\alpha_1 + \alpha_2}\cap (V)_2$ has dimension $2$.
\end{example}

\subsection{Coproduct and root vectors}
Let $\alpha\in \Phi_0^+$ and $r > 0$. Set $k_\alpha = \prod_{1\leq
  i\leq n}k_i^{\alpha(\omega_i^\vee)}$ We have \cite[Theorem 4,
(3)]{da}:
\begin{equation}\label{copro}\Delta(F_{-\alpha + r\delta}) \in
  F_{-\alpha + r\delta}\otimes k_\alpha + \sum_{\beta \in \Phi_0^+, p > 0} U_q(\bo^-)\otimes
  (U_q(\bo^-)F_{-\beta + p\delta}).\end{equation} 
This gives the factor $k_\alpha$ in the decomposition of $F_{-\alpha +
  r\delta}$ in $U_q(\bo^-)^+ \otimes \mathfrak{t}$: 
\begin{equation}\label{factor}F_{-\alpha + r\delta} \in U_q(\bo^-)^+
  k_\alpha\text{ and }E_{-\alpha + r\delta} \in U_q(\bo)^-
  k_\alpha^{-1}.\end{equation}
This last point also follows from \cite[Proposition 9.3]{da2}. Now let
$i\in I$ and $r > 0$. The $Q$-grading of $U_q(\Glie)$ induces
a $Q$-grading of $U_q(\bo^-)$.  Let $U_q(\bo^-)_+$
(resp. $U_q(\bo^-)_-$) be the subalgebra of $U_q(\bo^-)$ of elements
of positive (resp. negative) $Q$-degree. Then we have
\begin{equation}\label{copro2}\Delta(h_{i,-r}) \in h_{i,-r} \otimes 1 + 1\otimes h_{i,-r}
 + (U_q(\bo^-))_-\otimes (U_q(\bo^-))_+.\end{equation}

\subsection{Drinfeld relations}

Let us give examples of relations between Drinfeld generators that we
will use:
\begin{equation}\label{drinfeldun}[x_{j,-m}^+,x_{i,0}^-] =
  \delta_{i,j}\frac{\delta_{m,0}k_i - \phi_{i,-m}^-}{q_i -
    q_i^{-1}}\text{ for $m\in\ZZ$,  $i,j\in I$,}\end{equation}
and for $m\geq 0$, $p\in\ZZ$, $i,j\in I$:
\begin{equation}\label{drinfelddeux}\phi_{i,-m}^- x_{j,p}^+ 
= -\sum_{0\leq l\leq m - 1}
q_i^{- lC_{i,j}}x_{j,p-l-1}^+ \phi_{i,-m+l+1}^-
+ \sum_{0\leq l\leq m}q_i^{- (l+1) C_{i,j}} x_{j,p-l}^+\phi_{i,-m+l}^-.\end{equation}
We will also use the following technical result.

\begin{lem}\label{decompxi} Let $i\in I$, $\alpha\in\Phi_0^+$, $r >
  0$. Then we have a decomposition in $U_q(\bo^-)$
$$x_{i,0}^-F_{-\alpha + r\delta} = q^{(\alpha,\alpha_i)}F_{-\alpha +
  r\delta}x_{i,0}^- + \sum_{-r\leq p\leq 0} a_p (\phi_{i,p}^- +
\phi_{i,p+1}^- + \cdots + \phi_{i,0}^-)k_\alpha + a k_i k_\alpha$$
where $a \in U_q(\mathfrak{b}^-)^{ > 0}$ has $\ZZ$-degree $-r$ and
$a_p\in U_q(\mathfrak{b}^-)^{ > 0}$ is a sum of elements of
$\ZZ$-degree $-r-p$ or $-r-p+1$.
\end{lem}

\begin{proof} We will compute the decomposition of $F_{-\alpha +
    r\delta}x_{i,0}^-$ in (\ref{triandecomp2}) by using the full
  quantum loop algebra $U_q(\g)$ and the decomposition
  (\ref{fdecomp}).  Indeed, by (\ref{factor}), $F_{-\alpha + r\delta}$
  is a product $x^+ k_\alpha$ where $x^+$ has degree $-r$ and is an
  algebraic combination of the $x_{j,-m}^+$ ($j\in I$, $0\leq m\leq
  r$).  We first have $F_{-\alpha + r\delta}x_{i,0}^- =
  q^{-\lambda}x^+x_{i,0}^-k_\alpha$ where
$$\lambda = \sum_{j\in I}\alpha(\omega_j^\vee)d_j C_{j,i} = \sum_{j\in
  I} \alpha(\omega_j^\vee) (\alpha_i,\alpha_j) = (\alpha,\alpha_i).$$
The relations (\ref{drinfeldun}) imply
$$[x^+,x_{i,0}^-]\in \mathcal{A} k_i \mathcal{A} + \sum_{-r\leq m\leq
  0}  \mathcal{A} \phi_{i,m}^- \mathcal{A} .$$ 
where $\mathcal{A} = \CC[x_{j,p}^+]_{j\in I, 0\leq p\leq r}$. Then the
relations (\ref{drinfelddeux}) imply
$$[x^+,x_{i,0}^-]\in \mathcal{A} k_i + \sum_{-r\leq m\leq 0}  \mathcal{A} \phi_{i,m}^-.$$
Hence
\begin{equation}\label{thedecomp}F_{-\alpha + r\delta}x_{i,0}^- =
  q^{-\lambda} x_{i,0}^- F_{-\alpha + r\delta} - \sum_{-r\leq m\leq 0}
  q^{-\lambda}b_m \phi_{i,m}^- k_\alpha- q^{-\lambda} a k_i
  k_\alpha,\end{equation}
where $b_m\in \mathcal{A}$ has $\ZZ$-degree $-r-m$ and $a\in
\mathcal{A} $ has $\ZZ$-degree $r$. Note that $\mathcal{A}$ is not
contained in $U_q(\bo^-)$. But we have $\mathcal{A}\subset
U_q(\g)^+$. So (\ref{thedecomp}) is the decomposition of $F_{-\alpha +
  r\delta}x_{i,0}^-$ in (\ref{fdecomp}). But $F_{-\alpha +
  r\delta}x_{i,0}^-\in U_q(\bo^-)$ and the decomposition in
(\ref{triandecomp2}) is unique. Hence, for degree reason, we have
$a\in U_q(\bo^-)^+$ and $b_m\in U_q(\bo^-)^+$ for $m\leq 0$.  Now
(\ref{thedecomp}) can be rewritten as in the Lemma, with
$$a_p = b_{p} - b_{p-1}\text{ for $-r\leq p\leq 0$}$$
where we set $b_{-r-1} = 0$. 
\end{proof}

\begin{example} For example, in the case $\gb = sl_2$, we have for $r
  > 0$, $F_{r\delta-\alpha_1} = -x_{1,-r}^+k_1$ and
$$x_{1,0}^- F_{r\delta-\alpha_1} = q^2 F_{r\delta-\alpha_1} x_{1,0}^-
+ \phi_{1,-r}^-k_1.$$
\end{example}

\subsection{Tensor product of $\ell$-weight vectors} By using
(\ref{copro2}), we prove exactly as in \cite[Proposition 3.2]{h3} the
following.
\begin{prop}\cite{h3}\label{prodlweight}
  Let $V_1, V_2$ in category $\overline{\mathcal{O}}$ and consider an
  $l$-weight vector
$$w = \left(\sum_{\alpha} w_\alpha\otimes v_\alpha\right) +
\left(\sum_\beta w_\beta'\otimes v_\beta'\right)\in V_1\otimes V_2$$
satisfying the following conditions.

(i) The $v_\alpha$ (resp. $v_{\beta}'$) are $\ell$-weight
(resp. weight) vectors of $\ell$-weight $\Psib_\alpha$ (resp. weight
$\omega_\beta$).

(ii) For any $\beta$, there is an $\alpha$ satisfying $\omega_\beta >
\varpi(\Psib_\alpha)$.

(iii) For any $\alpha$, we have $\sum_{\left\{\alpha'|\omega_{\alpha'}
    = \omega_\alpha \right\}} w_{\alpha'}\otimes v_{\alpha'}\neq 0$.

\noindent Then the $\ell$-weight of $w$ is the product of one
$\Psib_\alpha$ by an $l$-weight of $V_1$.
\end{prop}

\subsection{Proof of Theorem \ref{defigrad}}
We can assume $a = 1$. By using the twisting by $\hat{\omega}$, we can
work with $\overline{L}_{i,1}^+$.  It will be important for our proof
as we will use that $U_q(\bo^-)^-$ as, in opposition to $U_q(\bo)^-$,
is generated by a family of Drinfeld generators (see
(\ref{negger})). We do not need such a property for $U_q(\bo^-)^+$ as
we have the coproduct formulas (\ref{copro}). The conditions to be
proved become for $m\geq 0$, $j\in I$, $r\leq 0$:

(1) for $x\in U_q(\bo^-)^+$ of degree $r < 0$,
$x((\overline{L}_{i,1}^+)_{m})\subset (\overline{L}_{i,1}^+)_{m + r}$.

(2) $\phi_{j,r}^- (\overline{L}_{i,1}^+)_{m}\subset
(\overline{L}_{i,1}^+)_{m + r}\text{ if $j\neq i$}$ and $(\phi_{i,r}^-
+ \phi_{i,r+1}^- + \cdots +
\phi_{i,0}^-)(\overline{L}_{i,1}^+)_{m}\subset
(\overline{L}_{i,1}^+)_{m + r}$.

(3) $x_{j,r}^-((\overline{L}_{i,1}^+)_{m})\subset
(\overline{L}_{i,1}^+)_{m + r} + (\overline{L}_{i,1}^+)_{m + r +
  \delta_{i,j}}$.

\noindent Let $v$ (resp. $v'$) be an $\ell$-highest weight vector of
$\overline{L}_{i,1}^+$ (resp. of $\tilde{L}(Y_{i,q_i^{-1}})$).  Let
$$V = U_q(\bo^-).(v\otimes v')\subset \overline{L}_{i,1}^+\otimes \tilde{L}(Y_{i,q_i^{-1}})$$
Then we have a surjective morphism of $U_q(\mathfrak{b}^-)$-modules 
$$\phi : V\rightarrow \overline{L}_{i,q_i^2}^+.$$ 
Let $V' = \overline{L}_{i,1}^+\otimes v'$. From (\ref{copro}), (\ref{copro2}), 
$V'$ is a $U_q(\mathfrak{b}^-)^{\geq 0}$-module and 
$$\chi_q(V') = \chi_q(\overline{L}_{i,q_i^2}^+).$$
Let us prove that $V'\subset V$. Consider an $\ell$-weight vector $w$
of $V$ whose $\ell$-weight is an $\ell$-weight of
$\overline{L}_{i,q_i^2}^+$. If $w$ is not in $V'$, in a decomposition
of $w$ as in Proposition \ref{prodlweight}, we would have some terms
$w_\alpha\otimes v_\alpha$ with $v_\alpha$ $\ell$-weight vector of
$\tilde{L}(Y_{i,q_i^{-1}})$ which is not in $\CC. v'$. But by
\cite[Lemma 6.1, Remark 6.2]{Fre2} (and its proof), we have
$$\tilde{\chi}_q(\tilde{L}(Y_{i,q_i^{-1}}))\in 1 + A_{i,1}^{-1}\ZZ[A_{j,b}^{-1}]_{j\in I, b\in\CC^*}.$$
Hence the $\ell$-weight $\Psib_\alpha$ of $v_\alpha$ would be in
$([\overline{\omega_i}]^{-1}Y_{i,q_i^{-1}})A_{i,1}^{-1}\ZZ[A_{j,b}^{-1}]_{j\in I, b\in\CC^*}$.  Contradiction as
by Theorem \ref{formuachar}, $A_{i,1}^{-1}$ is not a factor of the
$\ell$-weights occurring in $\tilde{\chi}_q(\overline{L}_{i,q_i^2}^+)$. Moreover,
$\overline{L}_{i,1}^+$ and $\overline{L}_{i,q_i^{-2}}^+$ have the same
character, so it implies that $V'\subset V$.

Now we may consider the restriction of $\phi$ to $V'$. From our
discussion, it is an isomorphism of $U_q(\mathfrak{b}^-)^{\geq
  0}$-module. It induces a linear isomorphism
$$\tilde{\phi} : \overline{L}_{i,1}^+\rightarrow \overline{L}_{i,q_i^2}^+.$$ 
Let $\tau : \overline{L}_{i,q_i^2}^+\rightarrow \overline{L}_{i,1}^+$
be the unique linear isomorphism such that $\tau(g.x) =
q_i^{-2m}g.\tau(x)$ for any $x\in \overline{L}_{i,q_i^2}^+$, $g\in
U_q(\mathfrak{b}^-)$ of $\ZZ$-degree $m$ and such that $\Phi(v) = v$
where
$$\Phi = \tau\circ \tilde{\phi} : \overline{L}_{i,1}^+\rightarrow \overline{L}_{i,1}^+.$$ 
We have for $j\in I$, $r > 0$ and $\alpha\in\phi_0^+$
$$\tilde{\phi} F_{-\alpha + r\delta} = F_{-\alpha + r\delta}\tilde{\phi}\quad,\quad
\tilde{\phi} \phi_j^-(z) = \phi_j^-(z)\left(\frac{1 - z^{-1} }{1 -
    z^{-1}q_i^2}\right)^{\delta_{i,j}} \tilde{\phi}.$$ Hence $\Phi$ is
a linear automorphism of $\overline{L}_{i,1}^+$ which commutes with
the $k_j$ and
$$\Phi F_{-\alpha + r\delta} = q_i^{2r} F_{-\alpha + r\delta}\Phi\quad ,\quad
\Phi \phi_j^-(z)(1- z^{-1})^{-\delta_{i,j}} = \phi_j^-(q_i^{-2} z)(1 -
(z q_i^{-2})^{-1})^{-\delta_{i,j}} \Phi.$$ For $i = j$, the last
equation can be rewritten as $(r \geq 0)$:
$$\Phi (\phi_{i,-r}^- + \phi_{i,-r+1}^- + \cdots + \phi_{i,0}^-) =
q_i^{2r}(\phi_{i,-r}^- + \phi_{i,-r+1}^- + \cdots +
\phi_{i,0}^-)\Phi.$$
The weight spaces of $\overline{L}_{i,1}^+$ are stable by $\Phi$. Let
us prove by induction on the height of $\alpha$ that $\Phi_{-\alpha}$
is diagonalizable on $(\overline{L}_{i,1}^+)_{-\alpha}$ with
eigenvalues of the form $q_i^{-2m}$, with $m\geq 0$ integer.  For
$\alpha = 0$ it is clear by construction. In general, there is a
finite family $(\alpha_1,r_1), \cdots, (\alpha_R,r_R)$ with the
$\alpha_j\in \Phi_0^+$, $r_j > 0$ such that the intersection of the
$\text{Ker}(F_{-\alpha_j + r_j\delta})\cap
(\overline{L}_{i,1}^+)_{-\alpha}$ is zero. By the induction
hypothesis, there is $M \geq 0$ such that the polynomial
$$P(X) = \prod_{0\leq m\leq M}(X - q_i^{-2m})$$ 
satisfies $P(\Phi) = 0$ on $\bigoplus_{1\leq j\leq R}
(\overline{L}_{i,1}^+)_{-\alpha+\alpha_j}$.  Let $r = Max_j (r_j)$ and
consider the polynomial
$$Q(X) = \prod_{0\leq m\leq M + r}(X - q_i^{-2m}).$$ 
For $1\leq j\leq R$ we have
$$Q(\Phi q_i^{-2r_j}) F_{-\alpha_j + r_j\delta} =  F_{-\alpha_j + r_j\delta}Q(\Phi).$$
But $P(X)$ divides $Q(X q_i^{-2r_j})$ and so $Q(\Phi q_i^{-2r_j}) = 0$
on $\bigoplus_{1\leq j\leq R}
(\overline{L}_{i,1}^+)_{-\alpha+\alpha_j}$.  Hence the operator
$F_{-\alpha_j + r_j\delta}Q(\Phi)$ is zero on
$(\overline{L}_{i,1}^+)_{-\alpha}$. Since this is true for any $j$, we
get $Q(\Phi) = 0$ on $(\overline{L}_{i,1}^+)_{-\alpha}$ and the
result.

For $m > 0$, we can define $(\overline{L}_{i,1}^+)_{m}$ as the
eigenspace of $\Phi$ of eigenvalue $q_i^{-2m}$.

Let us prove (3).  For $r < 0$ and $j\in I$, $[h_{j,r} -
\lambda_{j,r}, x_{j,0}^-] = [h_{j,r} , x_{j,0}^-]$ is a non zero
multiple of $x_{j,r}^-$. Hence, it suffices to prove the result for
$x_{j,0}^-$. If $j\neq i$, we have $x_{j,0}^-.v' = 0$ and so
$x_{j,0}^-\Phi = \Phi x_{j,0}^-$. Hence the result.  Now suppose that
$j = i$.  Consider a weight vector $w\in
(\overline{L}_{i,1}^+)_{m}$. We prove the result by induction on the
height of the weight of $w$. For $m = 0$ is follows from the
case $\gb = sl_2$. In general, by Section \ref{roots}, there is $F_{-\alpha
  + r\delta}$ ($\alpha \in \Phi_0^+$, $r > 0$) such that $F_{-\alpha +
  r\delta}x_{i,0}^-. w\neq 0$.  By the induction hypothesis
$$x_{i,0}^-F_{-\alpha + r\delta}.w\in (\overline{L}_{i,1}^+)_{m-r} + (\overline{L}_{i,1}^+)_{m-r+1}.$$ 
Let $\lambda, a_p, a, k$ as in Lemma \ref{decompxi}. By the result
above, we have for $p \leq 0$
$$a_p (\phi_{i,p}^- + \phi_{i,p+1}^- + \cdots + \phi_{i,0}^-)kw\subset
a_p (\overline{L}_{i,1}^+)_{m + p} \subset
(\overline{L}_{i,1}^+)_{m-r} + (\overline{L}_{i,1}^+)_{m-r+1}$$ as
$a_p\in U_q(\bo^-)^{> 0}$ is a sum of elements of $\ZZ$-degree $-r-p$
or $-r-p+1$. So
$$F_{-\alpha + r\delta}x_{i,0}^-. w \in (\overline{L}_{i,1}^+)_{m-r} + (\overline{L}_{i,1}^+)_{m-r+1}$$ 
and $x_{i,0}^-. w \in (\overline{L}_{i,1}^+)_{m} +
(\overline{L}_{i,1}^+)_{m+1}$.

To conclude, let us prove that the $(\overline{L_{i,1}}^+)_m$ are
finite-dimensional. First let us prove by induction on $m\geq 0$ that
$$(\overline{L}_{i,1}^+)_m =
\bigoplus_{\{\alpha|\alpha(\alpha_i^\vee)\leq m N_i\}}
(\overline{L}_{i,1}^+)_m \cap (\overline{L}_{i,1}^+)_{-\alpha}.$$ This
is clear if $m = 0$. For $m > 0$, let $w\in (\overline{L}_{i,1}^+)_m$
of weight $-\alpha$. Then there is $F_{-\beta + r\delta}$ such that
$F_{-\beta + r\delta}w\neq 0$.  We have $F_{-\alpha + r\delta}w \in
(\overline{L}_{i,1}^+)_{m-r}\cap
(\overline{L}_{i,1}^+)_{-\alpha+\beta}$. Hence, by induction
hypothesis, $(-\beta + \alpha)(\alpha_i^\vee)\leq (m-r)N_i$. But we
have $0\leq \beta(\alpha_i^\vee)\leq
N_i$. So $$\alpha(\alpha_i^\vee)\leq N_i + (m-r)N_i\leq m N_i$$ as $r
>0$.  If $m$ is fixed, by using the following Lemma \ref{nombre} there
is a finite number of weights $\alpha$ of $\overline{L}_{i,1}^+$ such
that $\alpha(\alpha_i^\vee)\leq m N_i$. Hence the result.

\begin{lem}\label{nombre} Let $i\in I$, $a\in\CC^*$ and $M\geq
  0$. There is a finite number of weights $\alpha$ of
  $\overline{L}_{i,a}^+$ such that $\alpha(\alpha_i^\vee)\leq M$.
\end{lem}

\begin{proof} By construction, the character of $\overline{L}_{i,a}^+$
  is the character of $L_{i,a}^-$ and it does not depend on $a$. It is
  proved in \cite[Theorem 6.1]{HJ} that $\chi(L_{i,1}^-)$ is the limit
  of $\tilde{\chi}(W_{k,1}^{(i)})$ when $k\rightarrow +\infty$. By
  (\ref{recu}), a weight $\alpha$ satisfying
  $\alpha(\alpha_i^\vee)\leq M$ is a weight of $\overline{L}_{i,a}^+$
  only if it is a weight of $W_{M,1}^{(i)}$. Hence the result.
\end{proof}

\begin{rem} In the proof, if instead of $\overline{L}_{i,q_i^{-2}}^+$
  we had used $\overline{L}_{i,q_i^{-2k}}^+$ with $k\geq 2$, we would
  get the same grading. \end{rem}

\section{End of the proof of Theorem \ref{catOpol}}\label{endproof}

In this section we finish the proof of Theorem \ref{catOpol}.
We first recall the factorization of the universal $R$-matrix 
and show that it implies Proposition \ref{defo}.
Then we establish the main reduction for the proof of Theorem \ref{catOpol}
in Section \ref{reduc}: it suffices to consider certain distinguished
tensor products of fundamental representations.

\subsection{Factorization of the universal $R$-matrix}\label{fact} The
universal $R$-matrix has a factorization \cite{da}
$$\mathcal{R} = \mathcal{R}^+\mathcal{R}^0\mathcal{R}^-\mathcal{R}^\infty,$$
where $\mathcal{R}^\pm \in U_q(\mathfrak{b})^\pm\hat{\otimes}
U_q(\mathfrak{b}^-)^\mp$, 
$$\mathcal{R}^0 = \text{exp}\left( -  \sum_{m > 0,i,j\in I} \frac{(q_i - q_i^{-1})(q_j -
    q_j^{-1})m
    \tilde{B}_{i,j}(q^m)}{(q - q^{-1})[m]_q}h_{i,m}\otimes h_{j,-m}\right),$$
and $\mathcal{R}^\infty = q^{-t_\infty}$ where
$t_\infty\in\dot{\mathfrak{h}}\otimes\dot{\mathfrak{h}}$ is the
canonical element (for the standard invariant symmetric bilinear form
as in \cite{da}), that is, if we denote formally $q = e^h$, then
$\mathcal{R}^\infty = e^{-h t_\infty}$.

For a variable $x$, the $q$-exponential in $x$ is a formal power series
$\text{exp}_{q^p} (x) = \sum_{r\geq 0} \frac{x^r}{[r]_{q^p}'!}$ where
$p\in\ZZ$ and $[r]_v'! = \prod_{1\leq s\leq r}\frac{v^{2s} - 1}{v^2 -
  1} = v^{\frac{r(r-1)}{2}}[r]_v!$ for $r\geq 0$.

$\mathcal{R}^+$ (resp. $\mathcal{R}^-$) is a product of
$q$-exponentials of a multiple of $E_{\alpha + m\delta}\otimes
F_{\alpha + m\delta}$ with $m \geq 0$, $\alpha\in\Phi_0^+$ (resp. with
$m > 0$, $\alpha\in\Phi_0^-$).

\begin{example} In the case $\gb = sl_2$, we have 
$$\mathcal{R}^+ = \prod_{m\geq 0}\text{exp}_q\left((q^{-1} -
  q)x_{1,m}^+\otimes x_{1,-m}^-\right)\text{ , } \mathcal{R}^- =
\prod_{m > 0}\text{exp}_q\left((q - q^{-1})k_1^{-1}x_{1,m}^-\otimes
  x_{1,-m}^+  k_1\right),$$
$$\mathcal{R}^0 = \text{exp}\left( - (q - q^{-1}) \sum_{m > 0}
  \frac{m}{[m]_q(q^m + q^{-m})}h_{1,m}\otimes h_{1,-m}\right).$$
\end{example}

\subsection{Proof of Proposition \ref{defo}}\label{proofdefo}
Let $i\in I$. Since each $F_{\alpha + m\delta}$ ($m > 0,
\alpha \in \Phi_0^-$) acts by $0$ on the lowest weight vector of
$R_{i,1}^+$, only $\mathcal{R}^0$ and $\mathcal{R}^\infty$ contribute
to the specialization of $t_{R_{i,1}^+}(z,u)$.  Let us replace in
$\mathcal{R}^0$ each $h_{l,m}$ ($m > 0, l\in I$) by $-z^m
\frac{\delta_{i,l}}{m(q_i - q_i^{-1})}$, that is we take the trace on
$R_{i,1}^+$. We get
$$\text{exp}\left(  \sum_{j\in I, m > 0} z^m
  \frac{\tilde{B}_{i,j}(q^m)[d_j]_q}{[m]_q}h_{j,-m}\right) =
\text{exp}\left( \sum_{m > 0} z^m
  \frac{\tilde{h}_{i,-m}}{[d_i]_q[m]_{q_i}}\right) = T_i(z)$$ 
  $$\text{as }[d_j]_q[m]_{q_j} = [m]_q [d_j]_{q^m}\text{ and
}\tilde{B}_{i,j}(q^m) = \tilde{C}_{j,i}(q^m)/[d_i]_{q^m}\text{ for any 
$m\in\ZZ$, $j\in I$.}$$ We have proved Proposition \ref{defo}.

\subsection{Reduction}\label{reduc} It suffices to prove theorem
\ref{catOpol} for $W$ in $\mathcal{C}$ simple or standard (that is a
tensor product of fundamental representations).  A representation in
$\mathcal{C}$ is said to be thin if its $\ell$-weight spaces are of
dimension $1$.
 
\begin{rem}\label{minuscule} Let $i\in I$ and $a\in\CC^*$. If $N_i =
  1$, then $L(Y_{i,a})$ is simple as a $U_q(\gb)$-module, as shown by
  V. Chari \cite{C1} (see also \cite[Remark 7.7]{HJ}). In particular,
  if this $U_q(\gb)$ fundamental representation is minuscule, then
  $L(Y_{i,a})$ is thin.
\end{rem}

\begin{prop} Suppose that $\gb$ is not of type $E_8$. Any simple
  object in $\mathcal{C}$ occurs as a simple constituent of a tensor
  product of thin fundamental representations.
\end{prop}

\begin{example} If $\gb = sl_2$, any simple object is $\mathcal{C}$ is
  isomorphic to a tensor product of thin KR-modules, as shown by
  V. Chari and A. Pressley \cite{CP}.
\end{example}

\begin{proof} This is clear for $\gb$ of type $A$, $B$, $C$ or $G_2$
  as all fundamental representations are thin \cite{hmon}.

  Type $D_n$ ($n\geq 4$): for $a\in\CC^*$ and $i \in\{1,n-1,n\}$, the
  representation $L(Y_{i,a})$ is thin (for example by Remark
  \ref{minuscule}).  We have to prove the result for $L(Y_{i,a})$,
  $2\leq i\leq n-2$.  It can be obtained by induction on $i$ by using
  the following relation for $1\leq i\leq n-3$:
$$[L(Y_{1,a})\otimes L(Y_{i,aq^{i+1}})] = [L(Y_{1,a}Y_{i,aq^{i+1}})] + [L(Y_{i+1,aq^i})].$$
This relation can be easily established following the proof of
$T$-systems relations in \cite{hcr}.  By \cite{Fre2}, it suffices to
prove that dominant monomials have the same multiplicities when the
$q$-character morphism is applied on both sides of the
identity. Besides, the $q$-character of a fundamental representation
has a unique dominant monomial (that is it is affine-minuscule) and
its $q$-character is given by the algorithm of Mukhin and the first
author \cite{Fre2}. Then it is not difficult to see that the left side
of the identity has $2$ dominant monomials: the highest monomial
$Y_{1,a}Y_{i,aq^{i+1}}$ and
$Y_{1,a}Y_{i,aq^{i+1}}A_{1,aq}^{-1}A_{2,aq^2}^{-1}\cdots
A_{i,q^i}^{-1} = Y_{i+1,aq^i}$.  Then the identity result follows as
$L(Y_{i+1,aq^i})$ is affine-minuscule and as it can be proved as in
\cite{hcr, hmin} that $L(Y_{1,a}Y_{i,aq^{i+1}})$ is affine-minuscule.

Type $F_4$: for $a\in\CC^*$ and $i = 1, 4$, the representation
$L(Y_{i,a})$ is thin \cite{hmon}. We have to prove the result for
$L(Y_{2,a})$, $L(Y_{3,a})$. As above, it follows from the relations for
$a\in\CC^*$:
$$[L(Y_{1,a})\otimes L(Y_{1,aq^4})] = [L(Y_{1,a}Y_{1,aq^4})] + [L(Y_{2,aq^2})],$$
$$[L(Y_{4,a})\otimes L(Y_{4,aq^2})] = [L(Y_{4,a}Y_{4,aq^2})] + [L(Y_{3,aq})].$$

Type $E_6$: for $a\in\CC^*$ and $i = 1,5$, the representation
$L(Y_{i,a})$ is thin (for example by Remark \ref{minuscule}). By using
the same arguments as above, this implies the result for $i \neq
6$. Then we have
$$[L(Y_{1,a})\otimes L(Y_{5,aq^6})] = [L(Y_{1,a}Y_{5,aq^6})] + [L(Y_{6,aq^3})].$$

Type $E_7$: for $a\in\CC^*$, the representation $L(Y_{6,a})$ is thin
(for example by Remark \ref{minuscule}). As above, we get the result
for $i = 6,5,4,3$.  Now the monomial
$$Y_{1,aq^5}Y_{6,aq^{10}}^{-1} = Y_{6,a}A_{6,aq}^{-1}A_{5,aq^2}^{-1}A_{4,aq^3}^{-1}A_{3,aq^4}^{-1}A_{7,aq^5}^{-1}A_{2,aq^5}^{-1}A_{3,aq^6}^{-1}A_{4,aq^7}^{-1}A_{5,aq^8}^{-1}A_{6,aq^9}^{-1}$$
occurs in $\chi_q(L(Y_{6,a}))$. In particular, we get as above 
$$[L(Y_{6,a})\otimes L(Y_{6,aq^{10}})] = [L(Y_{6,a}Y_{6,aq^{10}})] +
[L(Y_{1,aq^5})].$$
So we have the result for $i = 1, 2$. We conclude as
$$[L(Y_{6,a})\otimes L(Y_{1,aq^7})] = [L(Y_{6,a}Y_{1,aq^7})] + [L(Y_{7,aq^4})].$$
\end{proof}

\begin{rem} The statement is not satisfied for $\gb$ of type
  $E_8$. Indeed the fundamental representation $L(Y_{i,a})$ is not
  thin for any $i$ in this case.  For $i = 1$, it is known by
  \cite[Section 6.1.2]{hn}. The Lie algebra for the sub Dynkin diagram
  $I\setminus\{7\}$ is of type $D_7$. The $q$-character of fundamental
  representations are known for $\gb$ of type $D$.  In particular we have the
  result for $2\leq i\leq 5$. In the same way we conclude for $i = 6$
  by considering $I\setminus\{1,2,3\}$.  The Lie algebra for the sub
  Dynkin diagram $I\setminus\{1\}$ is of type $E_7$ and by using
  \cite[Section 6.1.2]{hn}, we get the result for $i = 7$.  The Lie
  algebra for the sub Dynkin diagram $I\setminus\{1,2\}$ is of type
  $E_6$ and by using \cite[Section 6.1.2]{hn}, we get the result for
  $i = 8$.
\end{rem}

\begin{prop} Suppose that $\gb$ is of type $E_8$. Any simple object in
  $\mathcal{C}$ occurs as a simple constituent of a tensor product of
  fundamental representations $L(Y_{i,a})$ with $i = 1$.
\end{prop}

\begin{proof} As above, we have the result for $1\leq i\leq 5$. Now
  the monomial
$$Y_{7,aq^6}Y_{1,aq^{12}}^{-1} =
Y_{1,a}A_{1,aq}^{-1}A_{2,aq^2}^{-1}A_{3,aq^3}^{-1}A_{4,aq^4}^{-1}A_{5,aq^5}^{-1}A_{8,aq^6}^{-1}
A_{6,aq^6}^{-1}A_{5,aq^7}^{-1}A_{4,aq^8}^{-1}A_{3,aq^9}^{-1}A_{2,aq^{10}}^{-1}A_{1,aq^{11}}^{-1}$$
occurs in $\chi_q(L(Y_{1,a}))$. In particular, we get as above
$$[L(Y_{1,a})\otimes L(Y_{1,aq^{12}})] = [L(Y_{1,a}Y_{1,aq^{12}})] + [L(Y_{7,aq^6})].$$
So we have the result for $i = 7, 6$. We conclude as
$$[L(Y_{1,a})\otimes L(Y_{7,aq^8})] = [L(Y_{1,a}Y_{7,aq^8})] + [Y_{8,aq^5}].$$
\end{proof}

Consequently, if $\gb$ is not type $E_8$ (resp. is of type $E_8$) it
suffices to prove Theorem \ref{catOpol} for $W$ tensor product of thin
fundamental representations (resp. of fundamental representations $L(Y_{1,a})$).

\subsection{Proof of Theorem \ref{catOpol}}\label{proofgene}

By (\ref{shifttr}), we can assume that $a = 1$. We use the grading of
$V = R_{i,1}^+$ established in Theorem \ref{defigrad}.  For $j\in I$
and $r > 0$ let us consider $\overline{h}_{j,r} = h_{j,r} -
\frac{\delta_{i,j}}{r (q_i - q_i^{-1})}$. Note that by Remark
\ref{nilph}, $\pi_V(\overline{h}_{j,r})$ is nilpotent.  We have
$$\mathcal{R}^0 =
 \text{exp}\left( - \frac{(q_i - q_i^{-1})(q_j - q_j^{-1})}{(q -
     q^{-1})} \sum_{m > 0,i,j\in I} \frac{m
     \tilde{B}_{i,j}(q^m)}{[m]_q}\overline{h}_{i,m}\otimes
   h_{j,-m}\right) T_i(1).$$
This implies a factorization
$$L_V(z) = L_V^+(z) L_V^0(z)  (\text{Id}_V\otimes T_i(z))L_V^-(z)L_V^\infty,$$
where $L_V^\pm(z) = (\pi_V(z)\otimes Id)(\mathcal{R}^\pm)$, $L_V^\infty =
q^{-(\pi_V\otimes \text{Id})(t_\infty)}$ and
$$L_V^0(z) = \text{exp}\left( - \frac{(q_i - q_i^{-1})(q_j -
    q_j^{-1})}{(q - q^{-1})} \sum_{m > 0,i,j\in I} z^m\frac{m
    \tilde{B}_{i,j}(q^m)}{[m]_q}\pi_V(\overline{h}_{i,m})\otimes
  h_{j,-m}\right).$$

\noindent Let $W$ be a simple object in $\mathcal{C}$. The image of
$t_V(z,u)$ in $(\text{End}(W))[[z,u_j^{\pm 1}]]_{j\in I}$ is a
(possibly infinite) linear combination of terms of the following form
(we do not include $L_V^\infty$ which does not depend on $z$): the
product of two factors, the first one being
\begin{equation}\label{termun}z^R Tr_{V,u}(E_{\beta_1 +
    s_1\delta}\cdots E_{\beta_{r'} +
    s_{r'}\delta}\overline{h}_{i_1,r_1}\cdots
  \overline{h}_{i_p,r_p}E_{q_1 \delta -\gamma_1  }\cdots E_{q_r\delta
    -\gamma_{r}  })\end{equation}
and the second one being
\begin{equation}\label{termdeux}\pi_W(F_{\beta_1+s_1\delta}\cdots
  F_{\beta_{r'}+s_{r'}\delta}h_{i_1,-r_1}\cdots h_{i_p,-r_p}
  T_i(z)F_{q_1\delta -\gamma_1  }\cdots F_{q_r\delta -\gamma_{r}
  }),\end{equation}
where $\beta_1,\cdots, \beta_{r'}\in \Phi_0^+$,
$\gamma_1,\cdots,\gamma_r\in\Phi_0^+$, $i_1,\cdots, i_p\in I$,
$s_1,\cdots, s_{r'}\geq 0$, $r_1,\cdots, r_p > 0$, $q_1,\cdots, q_r >
0$ and
$$R = (s_1 + \cdots + s_{r'}) + (r_1 + \cdots + r_p) + (q_1 + \cdots  + q_r).$$
Moreover, so that (\ref{termun}) is non zero, we must have for weight
reason
$$\beta_1+ \cdots + \beta_{r'} = \gamma_1 + \cdots +\gamma_r$$
and by the conditions in Theorem \ref{defigrad}
$$R \leq (\beta_1(\alpha_i^\vee) + \cdots + \beta_{r'}(\alpha_i^\vee)).$$
Hence we have
$$R\leq (\gamma_1(\alpha_i^\vee) + \cdots + \gamma_{r}(\alpha_i^\vee)).$$
So that (\ref{termdeux}) is non zero on $(W)_\lambda$, we must have
for weight reason
$$(\gamma_1(\alpha_i^\vee) + \cdots + \gamma_{r}(\alpha_i^\vee))\leq
ht_i(\omega - \lambda).$$
This implies $R\leq ht_i(\omega - \lambda)$. Hence, there is a finite
number of choices for the $s_1,\cdots,s_{r'}$, $r_1,\cdots, r_p$,
$q_1,\cdots, q_r$.  So the total sum is a finite linear combination of
those terms.

Now suppose that $W$ is simple and thin.  By Proposition \ref{firstpol},
the eigenvalues of $T_i(z)$
on $W_\lambda$ are of the form $f(z) Q(z)$ where $f(z)$ does not
depend on $\lambda$ and $Q$ is a polynomial of degree $ht_i(\omega -
\lambda)$.  The restriction of $(f(z))^{-1}T_i(z)$ on $W_\lambda$ is a
polynomial of degree $ht_i(\omega - \lambda)$. Hence each factor
(\ref{termdeux}) is the product of $f(z)$ by a polynomial of degree
$ht_i(\omega - \lambda) - r$. Hence $(f(z))^{-1}t_V(z,u)$ is a
polynomial in $z$ of degree at most $ht_i(\omega - \lambda)$.

Now consider a tensor product $W = W_1\otimes \cdots \otimes W_N$ of
thin simple representations $W_i$. By
using inductively the formula $(\text{Id}\otimes \Delta)(\mathcal{R})
= \mathcal{R}_{1,2}\mathcal{R}_{1,3}$ we have for $N\geq 2$
\begin{equation}\label{intere}(\text{Id}\otimes \Delta)^N(\mathcal{R})
  = \mathcal{R}_{1,2}\mathcal{R}_{1,3}\mathcal{R}_{1,4}\cdots
  \mathcal{R}_{1,N+2}.\end{equation}
Hence we can use the same proof as above where the term inside $\pi_V$
in the factor (\ref{termun}) is replaced by a product of $N$ such
terms and the factor (\ref{termdeux}) is replaced by a product of $N$
such terms with $\pi_{W_j}$ ($1\leq j\leq N$) instead of $\pi_W$.

Let us conclude the proof of Theorem \ref{catOpol} when $\gb$ is not
of type $E_8$.  It suffices to prove that the degree of the polynomial
in $z$ is not only less than or equal to $ht_i(\omega - \lambda)$, but
equal to it. First note that Theorem \ref{catOpol} implies Theorem
\ref{jordan}, so the degree of the polynomial in Theorem \ref{jordan}
is less than or equal to the degree in Theorem \ref{catOpol}. But the
eigenvalues of $(f_{i}(z))^{-1}T_i(z)$ are exactly of degrees
$ht_i(\omega - \lambda)$ by Proposition \ref{firstpol}. Hence the
result. For the same reason, the degree is exactly $ht_i(\omega -
\lambda)$ in Corollary \ref{catOcase}.

To conclude, let us prove the result when $\gb$ is of type $E_8$. We
have seen it suffices to consider the case of a tensor product of
fundamental representations $W = L(Y_{1,a})$. This representation has
been studied for example in \cite[Section 6.1.2]{hn}\footnote{The
  representation $W$ has dimension $249$. As a $U_q(\gb)$-module, it
  has two simple constituents, one of them is the trivial
  representation of dimension $1$.}. It has a unique $\ell$-weight
space $W'$ whose dimension is not $1$: it corresponds to the monomial
$m = Y_{5,aq^{14}}Y_{5,aq^{16}}^{-1}$ and it is of dimension $2$.  To
use the same argument as above, we just have to prove the action of
$(f(z))^{-1}T_i(z)$ on $W'$ is a polynomial of degree $ht_i(\omega -
\lambda)$ (with the same notation as above).

The $q$-character of $W$ can be computed by using the algorithm of
Mukhin and the first author
\cite{Fre2}. In particular the monomials in $m\ZZ[A_{5,b}^{\pm
  1}]_{b\in\CC^*}$ which occur in $\chi_q(W)$ are necessarily $m$, $m
A_{5,aq^{15}}$, $mA_{5,aq^{15}}^{-1}$ with respective multiplicities
$1$, $2$, $1$. Besides, by setting $Y_{j,b} = 1$ for $j\neq 5$ in $ m
A_{5,aq^{15}} + 2m + mA_{5,aq^{15}}^{-1}$ we get $Y_{5,aq^{14}}^2 + 2
Y_{5,aq^{14}}Y_{5,aq^{16}}^{-1} + Y_{5,aq^{16}}^{-2}$ which is the
$q$-character of a simple $U_q(\hat{sl_2})$-module of dimension $4$.
Let $U$ be the subalgebra of $U_q(\Glie)$ generated by the
$x_{5,m}^\pm$, $k_5^{\pm 1}$, $\tilde{h}_{5,r}$ ($m\in\ZZ$,
$r\in\ZZ\setminus\{0\}$). Then $U$ is isomorphic to $U_q(\hat{sl_2})$.
From the discussion above, $W'$ generates a simple $U$-module of
dimension $4$ which is the sum of $3$ $\ell$-weight spaces of $W$.
Let $w$ be an highest weight vector of this module. Then $\CC.w$ is a
$\ell$-weight space of $W$ of dimension $1$ and $(f(z))^{-1}T_i(z)$ is
a polynomial on $\CC.w$.  But we have also $W' = \sum_{m\in\ZZ}
\CC.x_{5,m}^-.w$. By Lemma \ref{commti}, for $i\neq 5$, $T_i(z)$
commutes with the
$x_{5,m}^-$. So, we have the result on $W'$. Suppose that $i = 5$.
Then $T_5(z)\in U[[z]]$. The result follows from Theorem \ref{jordan}
that we have already established in the case $\gb = sl_2$ (in fact, this is
exactly the example explained in Section \ref{secjordan}).

\begin{rem} In the case $N_i = 1$, for example for any $i$ for $\gb$
  of type $A$, the proof in Section \ref{proofgene} is simplified
  thanks to Section \ref{simpleun}.  Indeed, by Theorem \ref{asympt},
  we have $\pi_V(E_{\alpha + m\delta}) = 0$ for ($m\geq 1$ and $\alpha
  \in\Phi_0^+$) or ($m > - \alpha(\alpha_i^\vee)$ and $\alpha
  \in\Phi_0^-$). Moreover we have a scalar action $\pi_V(h_{j,m}) =
  \frac{-\delta_{i,j}\text{Id}_V}{m (q_i - q_i^{-1})}$ for $m > 0$ and
  $j\in I$.  This implies $L_V^0(z) =1$.  Consider a tensor product $W
  = W_1\otimes \cdots \otimes W_N$ of thin simple representations. By
  using formula (\ref{intere}), we can use the same arguments as above
  by considering only terms of the form
$$z^{r_1+\cdots + r_N} Tr_{V,u}((E_{\beta_1^{(1)}}\cdots
E_{\beta_{r'_1}^{(1)}}E_{\delta -\gamma_1^{(1)}}\cdots E_{\delta
  -\gamma_{r_1}^{(1)}})\cdots (E_{\beta_1^{(N)}}\cdots
E_{\beta_{r'_N}^{(N)}}E_{\delta -\gamma_1^{(N)}}\cdots E_{\delta
  -\gamma_{r_N}^{(N)}}))$$
$$\pi_{W_1}(F_{\beta_1^{(1)}}\cdots
F_{\beta_{r'_1}^{(1)}}T_i(z)F_{\delta -\gamma_1^{(1)}}\cdots F_{\delta
  -\gamma_{r_1}^{(1)}})\otimes \cdots \otimes
\pi_{W_N}(F_{\beta_1^{(N)}}\cdots
F_{\beta_{r'_N}^{(N)}}T_i(z)F_{\delta -\gamma_1^{(N)}}\cdots F_{\delta
  -\gamma_{r_N}^{(N)}})
$$
where $r_1\leq R_1,\cdots, r_N\leq R_N$,
$\gamma_1^{(1)}(\alpha_i^\vee) = \cdots =
\gamma_{r_N}^{(N)}(\alpha_i^\vee) = 1$ and $R_1,\cdots, R_N$ are
fixed.
\end{rem}

\begin{example} Let us now consider the example in Section \ref{sl2ex}
  from the angle of the proof above (strictly speaking, the example
  here is essentially equivalent to Section \ref{sl2ex}).  In the
  case $\gb = sl_2$ with $V = R_{1,1}^+$, we have $\pi_V(x_{1,m}^+) = 0$ for
  $m\geq 1$ and $\pi_V(x_{1,m}^-) = 0$ for $m\geq 2$.  Moreover
  $\pi_V(h_{1,m}) = \frac{-\text{Id}_Va^m}{m (q - q^{-1})}$ for $m >
  0$. Hence
$$L_V(z) = L_V^+(z)(\text{Id}_V\otimes T(az))L_V^-(z) (\pi_V\otimes
\text{Id}) (\mathcal{R}^\infty),$$
$$L_V^+(z) = \text{exp}_q\left((q^{-1} - q)\pi_V(x_{1,0}^+)\otimes
  x_{1,0}^-\right)\text{ , }
L_V^-(z) = \text{exp}_q\left((q -
  q^{-1})z\pi_V(k_1^{-1}x_{1,1}^-)\otimes x_{1,-1}^+ k_1\right).$$ Hence
$$t_V(z,u) = \sum_{r\geq 0} \frac{((q -
  q^{-1})z)^r}{[r]_q!}  \sum_{m\geq r}
u^{2m}\begin{bmatrix}m\\r\end{bmatrix}_q
q^{\frac{r(3-r)}{2}-rm} 
(x_{1,0}^-)^r T_1(z)
(x_{1,-1}^+k_1)^r k_1^{-m}.$$
When we specialize $x_{1,-1}^+$ on a simple finite-dimensional
representation $W$, it becomes nilpotent and the sum is finite. Then
we recover the formulas as in Section \ref{sl2ex}.
\end{example}

\end{document}